\documentclass[a4paper,11pt]{amsart}

\usepackage[pdftex]{graphicx}  
\usepackage{amsmath,amssymb}
\usepackage{amsthm}
\usepackage{mathrsfs}
\usepackage{enumitem}
\usepackage{mathtools}
\usepackage[all,cmtip,2cell]{xy}
\usepackage{array}
\usepackage{here}
\usepackage{comment}
\usepackage[flushleft]{threeparttable}
\makeatletter
\let\MYcaption\@makecaption
\makeatother

\usepackage{subcaption}
\makeatletter
\let\@makecaption\MYcaption
\makeatother

\theoremstyle{plain}
\newtheorem{thm}{Theorem}[section]
\newtheorem{prop}[thm]{Proposition}
\newtheorem{lem}[thm]{Lemma}
\newtheorem{cor}[thm]{Corollary}
\newtheorem*{thm*}{Theorem}
\newtheorem{prob}{Problem}[section]

\theoremstyle{definition}
\newtheorem{defi}[thm]{Definition}

\newtheorem{assu}{Assumption}
\newtheorem{nota}[thm]{Notation}
\newtheorem*{nota*}{Notation}

\theoremstyle{remark}

\newtheorem{ex}[thm]{Example}

\numberwithin{equation}{subsection}

\newcommand{\Z}{\mathbb{Z}}
\newcommand{\Q}{\mathbb{Q}}

\newcommand{\C}{\mathbb{C}}
\newcommand{\A}{\mathbb{A}}
\renewcommand{\P}{\mathbb{P}}
\newcommand{\F}{\mathbb{F}}

\newcommand{\vp}{\varphi}

\newcommand{\mc}{\mathcal}

\newcommand{\ol}{\overline}
\newcommand{\ul}{\underline}

\newcommand{\wt}{\widetilde}

\DeclareMathOperator{\Coker}{Coker}

\DeclareMathOperator{\Aut}{Aut}

\DeclareMathOperator{\Pic}{\mathrm{Pic}}

\DeclareMathOperator{\eu}{\chi_{\mathrm{top}}}

\DeclareMathOperator{\Sing}{\mathrm{Sing}\:}

\DeclareMathOperator{\ra}{\rangle}
\DeclareMathOperator{\la}{\langle}

\title[Compactifications in $\mathrm{Bl}_C\P^3$ with  $K+D_{1}+D_{2}=0$]
{Compactifications of affine homology $3$-cells into blow-ups of the projective $3$-space \\with trivial log canonical divisors}
\author[M. NAGAOKA]{Masaru Nagaoka}
\address{Graduate School of Mathematical Sciences\\The University of Tokyo\\3-8-1 Komaba\\Meguro-ku, Tokyo 153-8914, Japan}
\email{nagaoka@ms.u-tokyo.ac.jp}
\begin{document}
\begin{abstract}
In this paper we classify all the compactifications of affine homology $3$-cells into the blow-ups of the projective $3$-space along smooth curves such that the log canonical divisors are linearly trivial.
As a result, we prove that each embedded affine $3$-fold is isomorphic to the affine $3$-space except one example.
\end{abstract}
\maketitle
\tableofcontents
\setcounter{section}{0}
\section{Introduction}\label{sec:1}

Throughout the paper we work over the field of complex numbers $\C$.
In \cite{Nag1}, we investigate about the following problem:

\begin{prob}[{\cite{Kis}}]\label{Q:main}
Let $V$ be a smooth Fano 3-fold with the second Betti number $B_{2}(V)=2$, $U$ a contractible affine $3$-fold which is embedded into V, and $D_{1}$ and $D_{2}$ irreducible hypersurfaces such that $V \setminus U = D_{1} \cup D_{2}$.  Classify such triplets $(V, U, D_{1} \cup D_{2})$.
\end{prob}

We also showed that there are exactly $14$ deformation equivalence classes of $V$ as in Problem \ref{Q:main} when the log canonical divisor $K_{V}+D_{1}+D_{2}$ is linearly trivial. 
Among them, exactly six equivalence classes parametrize the blow-ups of the projective $3$-space $\P^{3}$ along smooth curves.

The aim of this paper is to give the complete solution to Problem \ref{Q:main} when $(V, U, D_{1} \cup D_{2})$ as in Problem \ref{Q:main} satisfies the following conditions:
\begin{enumerate}
\item $K_V + D_{1} + D_{2} \sim 0$.
\item There is a blow-up morphism $\vp \colon V \to \P^3$ along a smooth curve $C \subset \P^3$.
\end{enumerate}
In the above situation, \cite[Proposition 3.4]{Nag1} shows that we may assume that $\vp_{*} D_{1}$ is a hyperplane not containing $C$ and $\vp_{*} D_{2}$ is a cubic surface containing $C$ with multiplicity one.
This assertion still holds if we drop the assumption that $V$ is Fano $3$-fold, and if we replace the contractibility of $U$ as that $U$ is \textit{an affine homology $3$-cell}, i.e., a smooth affine $3$-fold with $H_{i}(U, \Z) = 0$ for $i \geq 1$ (see Lemma \ref{lem:linequiv}).
For this reason, we consider the following problem instead of Problem \ref{Q:main}:

\begin{prob}\label{prob:main}
Let $C \subset \P^{3}_{[x:y:z:t]}$ be a smooth curve, $S_{1}$ a hyperplane not containing $C$, and $S_{2}$ a cubic surface containing $C$ with multiplicity one.
Take $\vp \colon V \rightarrow \P^{3}$ as the blow up along $C$ and write $D_{i}$ as the strict transform of $S_{i}$ in $V$ for $i=1,2$. 
Write $U \coloneqq V \setminus (D_{1} \cup D_{2}).$
Then classify $(C, S_{1}, S_{2}, U)$ such that $U$ is an affine homology $3$-cell. 
\end{prob}

The following example shows that $U$ as in Problem \ref{prob:main} may not be isomorphic to $\A^{3}$ even if $U$ is an affine homology $3$-cell.

\begin{ex}\label{ex:nonA3}
In $\P^{3}_{[x:y:z:w]}$, let $C=\{x=y=-z\}$, $S_{1}=\{z=0\}$ and $S_{2}=\{x^{2}z+y^{3}=0\}$. 
Take $\vp \colon V \rightarrow \P^3$ as the blow up along $C$ and write $D_{i}$ as the strict transform of $S_{i}$ in $V$ for $i=1,2$.
Then $V$ is the Fano $3$-fold of No.\ 33 in \cite[Table 2]{M-M} and $U \coloneqq V \setminus (D_{1} \cup D_{2})$ is isomorphic to $\A^{1} \times W(3, 2)$, where
\begin{align}
W(3, 2)  \coloneqq \left\{\frac{(zx+1)^2-(zy+1)^3-y}{y}=0\right\} \subset \A^{3}_{(x,y,z)}
\end{align}
is a contractible affine surface with the logarithmic Kodaira dimension one constructed in \cite{tDP}. In particular, $\A^{1} \times W(3,2) \not \cong \A^{3}$ because the Zariski cancellation problem has an affirmative answer in dimension two \cite{Fuj79, M-S80}.
\end{ex}

Our main result consists of two theorems.
One is Theorem \ref{thm:main3tuple}, which determines all the $3$-tuple $(C, S_{1}, S_{2})$ as in Problem \ref{prob:main} when $U$ is an affine homology $3$-cell.
Throughout the statement of Theorem \ref{thm:main3tuple}, $G_{i}$ and $R_{i}$ for $i \geq 1$ denote the cubic surfaces defined as in Definition \ref{defi:normrat} and Theorem \ref{thm:nnorm1} respectively.

\begin{thm}\label{thm:main3tuple}
We use the notation as in Problem \ref{prob:main}.
Then $U$ is an affine homology $3$-cell if and only if one of the following holds:
\begin{enumerate}
\item[\textup{(a)}] 
The curve $C$ is a smooth elliptic curve of degree three or four.
The surface $S_{2}$ is the cone over an elliptic curve whose vertex $S_{1}$ contains.
Moreover, $\sharp (C \cap S_{1})=B_{2}(S_{1} \cap S_{2})$.

\item[\textup{(b)}] The curve $C$ is a smooth rational curve of degree three or four.
The pair $(S_{1}, S_{2})$ is projectively equivalent to $(\{y= \gamma x\}, R_{1})$ for some $\gamma \in \C$.
Moreover, $\sharp (C \cap S_{1})=B_{2}(S_{1} \cap S_{2})$.

\item[\textup{(c)}] The curve $C$ is a smooth rational curve of degree three or four.
The pair $(S_{1}, S_{2})$ is projectively equivalent to $(\{y= \gamma x\}, R_{2})$ for some $ \gamma \in \P^{1}$.
Moreover, $\sharp (C \cap S_{1})=B_{2}(S_{1} \cap S_{2})+1$.

\item[\textup{(d)}] The curve $C$ is a smooth rational curve and $(S_{1}, S_{2})$ is projectively equivalent to one of the following:
\begin{align*}
&(\{y=0\}, G_{1}), (\{z=\gamma y\}, G_{5})\text{ for some } \gamma \in \P^{1},\\
&(\{t=0\}, G_{6}), (\{y=0\}, G_{9}), (\{y=0\}, G_{10}), (\{x=t\}, G_{11}),\nonumber\\
&(\{x=0\}, R_{1}), (\{y=\gamma x\}, R_{3}) \text{ for some } \gamma \in \P^{1}, \text{and } (\{x=0\}, R_{4}).\nonumber
\end{align*} 
Moreover, $\sharp(C \cap S_{1})=1$. 
\label{case:main-2}

\item[\textup{(e)}] The curve $C$ is a smooth rational curve and $(S_{1}, S_{2})$ is projectively equivalent to one of the following:
\begin{align*}
 (\{y=0\}, G_{2}), (\{y=0\}, G_{4}), \text{and } (\{y= \gamma x\}, R_{4}) \text{ for some } \gamma \in \A^{1}.
\end{align*}
Moreover, the inclusion $C \setminus (C \cap S_{1}) \hookrightarrow S_{2} \setminus (S_{2} \cap S_{1})$ induces an isomorphism $H_{1}(C \setminus (C \cap S_{1}), \Z) \cong H_{1}(S_{2} \setminus (S_{2} \cap S_{1}), \Z)$.

\item[\textup{(f)}] The triplet $(C, S_{1}, S_{2})$ is projectively equivalent to the subvarieties constructed as in Example \ref{ex:nonA3}.
\end{enumerate}
\end{thm}

The other is Theorem \ref{thm:mainisom}, which determines the isomorphism classes of $U$ as in Problem \ref{prob:main}.

\begin{thm}\label{thm:mainisom} 
We use the notation as in Problem \ref{prob:main}. 
Suppose that one of the cases \textup{(a)}--\textup{(e)} of Theorem \ref{thm:main3tuple} holds.
Then $U \cong \A^{3}$.
\end{thm}
Hence we obtain the following as a corollary of Theorems \ref{thm:main3tuple} and \ref{thm:mainisom}:
\begin{cor}
In the notation as in Problem \ref{prob:main}, 
$U$ is an affine homology $3$-cell if and only if $U \cong \A^{3}$ or $\A^{1} \times W(3, 2)$.
If the latter holds, then $(C, S_{1}, S_{2})$ is projectively equivalent to the subvarieties constructed as in Example \ref{ex:nonA3}.
\end{cor}

We now give an outline of the paper using the notation of Problem \ref{prob:main}.

\S \ref{sec:euler}--\S \ref{sec:curve} present some preliminaries.
In \S \ref{sec:euler}, we show that if the complement of a reduced member $D' \in |-K_{V}|$ is an affine homology $3$-cell, then we can determine the linear equivalence classes of irreducible components of $D'$.
After that, we set up notation and prove a lemma on topological Euler numbers.
In \S \ref{sec:cubicsurface}, we recall some facts on projective equivalence classes of cubic surfaces.
In \S \ref{sec:curve}, we summarize facts on curves in certain cubic surfaces.

We start proving main theorems from \S \ref{sec:ell}.

In \S \ref{sec:ell}, we prove Theorems \ref{thm:main3tuple} and \ref{thm:mainisom} when $C \not \cong \P^{1}$. 
In the remainder of this paper, we assume that $C \cong \P^{1}$.

In \S \ref{sec:normal} and \S \ref{sec:nonnormal}, we determine the projective equivalence class of $(S_{1}, S_{2})$ assuming that $U$ is an affine homology $3$-cell.
\S \ref{sec:normal} (resp.\ \S \ref{sec:nonnormal}) deals with the case where $S_{2}$ is normal (resp.\ non-normal).

In \S \ref{sec:cont}, we look closely at the contractibility of $U$ for each $(S_{1}, S_{2})$ which we determined in \S \ref{sec:normal} and \S \ref{sec:nonnormal}.
Combining the results of \S \ref{sec:ell}--\ref{sec:cont}, we complete the proof of Theorem \ref{thm:main3tuple}.
We also show that $U \cong \A^{3}$ when the case \textup{(d)} of Theorem \ref{thm:main3tuple} holds.

In \S \ref{sec:isomclass}, we prove that $U \cong \A^{3}$ when one of the cases \textup{(b)}, \textup{(c)} and \textup{(e)} of Theorem \ref{thm:main3tuple} holds to complete the proof of Theorem \ref{thm:mainisom}.

\begin{nota*}
Throughout this paper, we use the following notation:
\begin{itemize}
\item $\F_{d}$: the Hirzebruch surface of degree $d$.
\item $f_{d}$: a fiber of $\F_{d}$.
\item $\Sigma_{d}$: the minimal section of $\F_{d}$.
\item $\Q^{2}_{0}$: the quadric cone in $\P^{3}$.
\item $E_{f}$: the exceptional divisor of a birational morphism $f$.
\item $\Sing X$: the singular locus of a variety $X$.
\item $Y_{\wt X}$: the strict transform of a closed subscheme $Y$ of a normal variety $X$ in a birational model $\wt X$ of $X$.
\item $B_{i}(X)$: the $i$-th Betti number of a topological space $X$.
\item $\eu (X)$: the topological Euler number of a topological space $X$.
\item $\deg R$: the degree of a curve $R$ in the ambient projective space.
\item $p_{a}(R)$: the arithmetic genus of a projective curve $R$.
\item $\Aut(X)$: the automorphism group of $X$.
\end{itemize}
\end{nota*}

\section{Topological Euler numbers}\label{sec:euler}

In the beginning of this section, we confirm that the acyclicity of $U$ and the linear triviality of $K_{V}+D_{1}+D_{2}$ as in Problem \ref{Q:main} determine the linear equivalence classes of $D_{1}$ and $D_{2}$ even if $V$ is not Fano but the blow-up of $\P^{3}$ along a smooth curve.
After that, we set up notation and prove a lemma needed in the remainder of the paper.

\begin{lem}\label{lem:linequiv}
Let $C \subset \P^{3}$ be a smooth curve and $\vp \colon V \to \P^{3}$ the blow-up along $C$.
Let $D_{1}$ and $D_{2}$ be prime divisors such that $V \setminus (D_{1} \cup D_{2})$ is an affine homology $3$-cell and $K_{V}+D_{1}+D_{2} \sim 0$.
Then $D_{i} \sim \vp^{*} \mc{O}_{\P^{3}}(1)$ for some $i \in \{1,2\}$.
\end{lem}

\begin{proof}
For $i \in \{1, 2\}$, write $D_{i} \sim \vp^{*} \mc{O}_{\P^{3}}(a_{i})-b_{i}E_{\vp}$ with $a_{i}, b_{i} \in \Z$.
Then $a_{1}b_{2}-a_{2}b_{1}=\pm 1$ by \cite[Corollary 1.20]{Fuj}.
Since $K_{V}+D_{1}+D_{2} \sim 0$, we have $a_{1}+a_{2}=4$ and $b_{1}+b_{2}=1$.
On the other hand, let $L \subset \P^{3}$ be a line disjoint from $C$ and not contained in $\vp(D_{1} \cup D_{2})$.
Then $0 \leq (L_{V} \cdot D_{i})=a_{i}$ for $i \in \{1,2\}$.

If $a_{1}=0$, then  $\pm 1=-a_{2}b_{1}=-4b_{1}$, a contradiction.
Hence $a_{1} \neq 0$.
Similarly, we obtain $a_{2} \neq 0$.
On the other hand, if $a_{1}=a_{2}=2$, then $\pm 1=2(b_{2}-b_{1})$, a contradiction.
Hence $(a_{1}, a_{2})=(1, 3)$ or $(3,1)$.
If the former holds, then $\pm 1=a_{1}b_{2}-a_{2}b_{1}=1-4b_{1}$.
Therefore $b_{1}=0$ and $D_{1} \sim \vp^{*}\mc{O}_{\P^{3}}(1)$.
Similarly, if the latter holds, then $D_{2} \sim \vp^{*}\mc{O}_{\P^{3}}(1)$.
\end{proof}

\begin{nota}\label{not:2.8}
In the remainder of this paper, we use the notation as in Problem \ref{prob:main}.
We also use the following notation in addition:
\begin{itemize} 
\item $F \coloneqq S_{2}|_{S_{1}} \in |\mc{O}_{\P^{2}}(3)|$.
\item $\sigma \colon \ol S_{2} \to S_{2}$: the normalization.
\item $\tau \colon \wt S_{2} \to \ol S_{2}$: the minimal resolution.
\item $\mu \coloneqq \sigma \circ \tau$.
\item $E \subset S_{2}$ (resp.\ $\ol E \subset \ol S_{2}$): the conductor locus of $\sigma$.
\item $\wt{E} \coloneqq \ol E_{\wt S_{2}} \subset \wt S_{2}$.
\item $J_{i} \coloneqq \{ \vp$-exceptional curves in $ D_{i} \cap E_{\vp} \}, N_{i} \coloneqq \sharp J_{i}$ for $i=1,2$.
\item $J_{1 \cap 2} \coloneqq \{ \vp$-exceptional curves in $(D_{1} \cap D_{2}) \cap E_{\vp} \},  N_{1 \cap 2} \coloneqq \sharp J_{1 \cap 2}$.
\end{itemize}
We note that $E$, $\ol E$ and $\wt E$ are empty when $S_{2}$ is normal. 
We can also interpret $N_{1}$, $N_{2}$ and $N_{1 \cap 2}$ as $\sharp (C \cap S_{1})$, $\sharp (C \cap \Sing  S_{2})$ and $\sharp (C \cap S_{1} \cap \Sing  S_{2})$ respectively.
\end{nota}

\begin{lem}\label{lem:2.10}
If $U$ is an affine homology $3$-cell, then the following holds:
\begin{enumerate}
\item[\textup{(1)}] $\eu(F) = \eu(S_{2}) +2p_{a}(C)+ N_{1} +N_{2} - N_{1 \cap 2}-2$.
\item[\textup{(2)}] $N_{1} + N_{2} -N_{1 \cap 2} \geq 1$. 
Moreover, the equality holds if and only if $(N_{1}, N_{2}, N_{1 \cap 2})=(1,0,0)$ or $(1,1,1)$.
\end{enumerate}
\end{lem}

\begin{proof}
(1) For $i=1,2$, an easy computation shows that
\begin{align}\label{eq:2.10.1}
\eu(D_i)
&=\eu(D_i\setminus \bigcup_{L \in J_i}L)+\eu(\bigcup_{L \in J_i}L)\\
&=\eu(S_i\setminus \bigcup_{L \in J_i}\vp(L))+N_i\eu(\P^{1})\nonumber\\
&=\eu(S_i)-N_i+2N_i=\eu(S_i)+N_i.\nonumber
\end{align}
Substituting $\eu(S_{1})=\eu(\P^{2})=3$ into this, we obtain 
\begin{align}\label{eq:euler1}
\eu(D_{1})=N_{1}+3 \text{ and } \eu(D_{2}) = N_{2}+\eu(S_{2}).
\end{align}
In the same manner we obtain 
\begin{align}\label{eq:euler2}
\eu(D_{1} \cap D_{2})= \eu(F)+N_{1 \cap 2}.
\end{align}
On the other hand, the Mayer-Vietoris exact sequence gives
\begin{align}\label{eq:euler3}
\eu(D_{1} \cap D_{2})&= \eu(D_{1}) +\eu(D_{2})-\eu(D_{1} \cup D_{2}),\\
\eu(D_{1} \cup D_{2})& = \eu(V) -1=5-2p_{a}(C)\label{eq:euler4}
\end{align}
since $\eu(U)=1$.
Combining (\ref{eq:euler1})--(\ref{eq:euler4}), we get the assertion.

\noindent (2) The inequalities $N_{1}\geq 1$ and $N_{1 \cap 2} \leq \min\{N_{1}, N_{2}\}$ implies
the first assertion.
Now suppose that the equality holds. Then $0 \leq N_{2} - N_{1 \cap 2} = 1- N_{1} \leq 0$.
Hence $N_{2}=N_{1 \cap 2} \leq N_{1}=1$, which implies the second assertion.
\end{proof}

\section{Projective equivalence classes of cubic surfaces}\label{sec:cubicsurface}

In this section, we compile some results on projective equivalence classes of cubic surfaces in $\P^{3}$.

\subsection{Normal and rational cubic surfaces}\label{subsec:normalcubic}

We recall that cubic surfaces are canonically parametrized by $\P^{19}$ since $H^{0}(\P^{3}, \mc{O}_{\P^{3}}(3))=\C^{20}$, and there is an $\Aut(\P^{3})$-action on $\P^{19}$ such that its $\Aut(\P^{3})$-orbits correspond to projective equivalence classes of cubic surfaces.
Brundu-Logar \cite{B-L} roughly classified projective equivalence classes of normal and rational cubic surfaces as follows:

\begin{thm}[{\cite[Theorem 1.1]{B-L}}]\label{thm:2.11}
There is a finite disjoint union of quasi-projective varieties
\begin{align}
Q \coloneqq \bigsqcup_{i=1}^{13} T_i  \sqcup \bigsqcup_{(d,s) \in J} U_{d, s} \subset \P^{19}
\end{align}
such that each normal and rational cubic surface is projectively equivalent to the cubic surface corresponding to some point in $Q$, where
\begin{itemize}
\item $J \coloneqq \{(4, 0), (3, 1), (2, 2), (2, 1), (1, 3), (1, 2), (1, 1), (0, 4), (0, 3)\}$.
\item $T_{i}$ is a point for $i=1, \dots , 12$ and an open subset of a line in $\P^{19}$ for $i=13$. 
\item For each $(d, s) \in J$, $U_{d, s}$ is an open subset of a $d$-dimensional linear subspace in $\P^{19}$ which represents cubic surfaces having $s$ singular points and $r$ lines with $r=s+(d+2)(d+5)/2$. 
\end{itemize}
\end{thm}

\begin{defi}\label{defi:normrat}
Let $G_{i, p}$ denote the cubic surface which corresponds to a point $p$ in $T_{i}$ when $1 \leq i \leq 13$, in $U_{0, 3}$ when $i=14$ and in $U_{0,4}$ when $i=15$.
When $i \neq 13$ in addition, we often write it $G_{i}$ for short because $p$ is the unique point in $T_{i}$, $U_{0,3}$ or $U_{0,4}$.
\end{defi}

\begin{cor}\label{cor:B-L}
Each normal and rational cubic surface $S$ with $B_{2}(S) \leq 3$ is projectively equivalent to $G_{i}$ for some $i=1, \ldots, 12, 14, 15$, or $G_{13, p}$ for some $p \in T_{13}$.
\end{cor}

\begin{proof}
We may assume that $S$ corresponds to a point in $U_{d, s}$ for some $(d, s) \in J$.
Then $S$ contains exactly $s+(d+2)(d+5)/2$ lines.
Combining \cite[p.\ 255]{B-W} and $B_{2}(S) \leq 3$, we obtain $s+(d+2)(d+5)/2 \leq 9$. 
Hence $(d, s)=(0, 3), (0, 4)$ and $S$ is projectively equivalent to $G_{14}$ or $G_{15}$.
\end{proof}

For each $G_{i, p}$, Brundu-Logar also determined its defining equation as in Tables \ref{tb:norm} and its singularity as in Tables \ref{tb:sing}.
Moreover, we can write down all the lines in $G_{i, p}$ as in Table \ref{tb:line}.

\begin{table}[t]
\centering
\caption{The defining equations of $G_{i, p}$}
\label{tb:norm}
\begin{tabular}{|c|l|} \hline
$i=$         & The defining equation of $G_{i, p}$ in $\P^3_{[x,y,z,t]}$                         \\ \hline
$1$        & $xy^2+yt^2 +z^3$                                                              \\ \hline
$2$        & $xyt+xz^2+y^3$                                                                 \\ \hline
$3$        & $xyt-xzt+y^3$                                                                   \\ \hline
$4$        & $x^2y-x^2z-xy^2+xz^2+y^3-y^2t+yzt$                                   \\ \hline
$5$        & $x^2y+xz^2+y^2t$                                                               \\ \hline
$6$        & $xy^2+xyt+xzt+yt^2$                                                           \\ \hline
$7$        & $2x^2y-x^2z+xy^2-xyz-xyt-y^2t+yzt$                                   \\ \hline
$8$        & $x^2y-x^2z-2xy^2+2xyz-xyt-xz^2+xzt+y^2t$                          \\ \hline
$9$        & $x^2y-x^{2}z+xy^2-xyz+xz^2-y^2t$                                         \\ \hline
${10}$     & $x^2y-x^2z-2xyz+xz^2+y^2t$                                               \\ \hline
${11}$     & $x^2y +x^2z-xy^2+xyt-xzt-yt^2$                                          \\ \hline   
${12}$     & $x^2y+xyz-2xyt-xz^2+xzt-y^2t+yzt$                                     \\ \hline
${13}$     & $axy(x-t)+b(y-z)(x^2-xy-xz+2yt)$ for some $[a:b] \in \P^1$ \\ \hline
${14}$    & $x^2y-2x^2z-xy^2+xyz-y^2t+yzt+yt^2$                                  \\ \hline
${15}$    & $x^2y-xy^2+xz^2-yt^2$                                                       \\ \hline
\end{tabular}
\end{table}

\begin{table}[htpb]
\begin{threeparttable}
\centering
\caption{The singularity of $G_{i, p}$}
\label{tb:sing}
\begin{tabular}{|c|l|} \hline
$i=$         & The singularity of $G_{i, p}$  \\ \hline
$1$      & $([1\!:\!0\!:\!0\!:\!0],E_{6})$                                        \\ \hline
$2$      &  $([0\!:\!0\!:\!0\!:\!1],A_5), ([1\!:\!0\!:\!0\!:\!0],A_1)$                           \\ \hline
$3$      &  $([0\!:\!0\!:\!0\!:\!1],A_2), ([0\!:\!0\!:\!1\!:\!0],A_2), ([1\!:\!0\!:\!0\!:\!0],A_2)$                \\ \hline
$4$      &  $([0\!:\!0\!:\!0\!:\!1],A_5)$                                         \\ \hline
$5$      &  $([0\!:\!0\!:\!0\!:\!1],D_5)$                                         \\ \hline
$6$      &  $([0\!:\!0\!:\!1\!:\!0],A_4), ([1\!:\!0\!:\!0\!:\!0],A_1)$                           \\ \hline
$7$      &  $([0\!:\!0\!:\!0\!:\!1],A_2), ([0\!:\!1\!:\!1\!:\!0],A_2), ([0\!:\!0\!:\!1\!:\!0],A_1)$                  \\ \hline
$8$      &  $([0\!:\!0\!:\!0\!:\!1],A_3), ([1\!:\!0\!:\!0\!:\!1],A_1), ([0\!:\!0\!:\!1\!:\!1],A_1)$                \\ \hline
$9$      &  $([0\!:\!0\!:\!0\!:\!1],D_4)$                                         \\ \hline
${10}$   &  $([0\!:\!0\!:\!0\!:\!1],D_4)$                                         \\ \hline
${11}$   &  $([0\!:\!0\!:\!1\!:\!0],A_4)$                                         \\ \hline   
${12}$   &  $([0\!:\!0\!:\!0\!:\!1],A_1), ([0\!:\!1\!:\!1\!:\!0],A_3)$                              \\ \hline
${13}$   &  $([0\!:\!0\!:\!0\!:\!1],A_2), ([0\!:\!1\!:\!1\!:\!0],A_2)$                              \\ \hline
${14}$  &  $([0\!:\!1\!:\!1\!:\!0],A_2), ([0\!:\!0\!:\!1\!:\!0],A_1), ([0\!:\!0\!:\!1\!:\!-1],A_1)$                   \\ \hline
${15}$  &  $([1\!:\!0\!:\!0\!:\!1],A_1), ([0\!:\!1\!:\!1\!:\!0],A_1), ([1\!:\!0\!:\!0\!:\!-1],A_1), ([0\!:\!1\!:\!-1\!:\!0],A_1)$   \\ \hline
\end{tabular}
\begin{tablenotes}
\item The symbol $(\alpha, \beta)$ in the right column means that $\alpha \in G_{i, p}$ is a DuVal singularity of type $\beta$.
\end{tablenotes}
\end{threeparttable}
\end{table}

\begin{table}[htpb]
\begin{threeparttable}
\centering
\caption{The lines in $G_{i, p}$}
\label{tb:line}
\begin{tabular}{|c|l|} \hline
$i=$ & All the lines contained in $G_{i, p}$ \\ \hline
$1$     &  $ \la y,z \ra     $                                                                         \\ 
\hline
$2$     &  $ \la y,z \ra ,  \la x,y \ra $                                                                     \\ 
\hline
$3$     &  $ \la y,z \ra ,  \la x,y \ra ,  \la y,t \ra $                                                             \\ 
\hline
$4$     &  $ \la y,z \ra ,  \la x,y \ra ,  \la y,x-z \ra   $                                                       \\ 
\hline
$5$     &  $ \la y,z \ra ,  \la x,y \ra ,  \la x,t \ra ,  $                                                          \\ 
\hline
$6$     &  $ \la y,z \ra ,  \la x,y \ra ,  \la x,t \ra ,  \la y, t \ra  $                                                   \\ 
\hline
$7$     &  $ \la y,z \ra ,  \la x,y \ra ,  \la x,t \ra ,  \la x-t, y-z \ra ,  \la x, y-z \ra  $                                \\ 
\hline
$8$     &  $ \la y,z \ra ,  \la x,y \ra ,  \la x,t \ra ,  \la x-t, y-z \ra ,  \la y, x+z-t \ra    $                           \\ 
\hline
$9$     &  $ \la y,z \ra ,  \la x,y \ra ,  \la x,t \ra ,  \la x-t, y-z \ra ,  \la y, x-z \ra ,  $\\
                 &  $\la x-t, z-t \ra  $   \\ 
\hline                
${10}$  &  $ \la y,z \ra ,  \la x,y \ra ,  \la x,t \ra ,  \la x-t, y-z \ra ,  \la y, x-z \ra , $  \\
                 &  $ \la x-t, y-z+t \ra  $   \\ 
\hline            
${11}$  &  $ \la y,z \ra ,  \la x,y \ra ,  \la x,t \ra ,  \la x-y, y-t \ra ,  \la y,x-t \ra , $ \\
                 &  $  \la x-y,z+t \ra  $   \\ 
\hline                  
${12}$  &  $ \la y,z \ra ,  \la x,y \ra ,  \la x,t \ra ,  \la x-t, y-z \ra ,  \la y,z-t \ra ,  $\\
                 &  $ \la x, y-z \ra , \la x-t, y-z+t \ra  $   \\ 
\hline
${13}$  &  $ \la y,z \ra ,  \la x,y \ra ,  \la x,t \ra ,  \la x-t, y-z \ra ,  \la y,x-z \ra ,  $\\
                 & $\la x, y-z \ra ,  \la x-t, y-z+t \ra   $  \\ 
\hline
${14}$ &  $ \la y,z \ra ,  \la x,y \ra ,  \la x,t \ra ,  \la x-t, y-z \ra ,  \la x-y, z+t \ra ,  $ \\
             &  $\la x-y, z+t \ra , \la x-y, y-t \ra ,  \la x+t, y-z \ra $  \\ 
\hline
${15}$ &  $ \la y,z \ra ,  \la x,y \ra ,  \la x,t \ra ,  \la x-t, y-z \ra ,  \la x-y, z+t \ra ,  $  \\
              & $\la x-y, z-t \ra , \la x-t, y+z \ra ,  \la x+t, y-z \ra ,  \la x+t, y+z \ra $    \\ 
\hline
\end{tabular}
\begin{tablenotes}
\item The symbol $\la a, b \ra$ means the line defined by $a=b=0.$
\end{tablenotes}
\end{threeparttable}
\end{table}

\subsection{Non-normal cubic surfaces}\label{subsec:non-normal cubic}

Lee-Park-Schenzel \cite{lps11} classified non-normal cubic surfaces up to projective equivalence as follows:

\begin{thm}[{\cite[Theorem 3.1]{lps11}}]\label{thm:nnorm1}
Each non-normal cubic surface in $\P^3_{[x:y:z:t]}$ is projectively equivalent to one of the following:
\begin{align*}
R_{1} &=\{ x^2 z+y^3=0\},\\
R_{2} &= \{x^2 z +y^3 +y^2 z=0\},\\
R_{3} &= \{ x^2 z + y^3 +y^2 t=0\},\\
R_{4} &= \{x^2z + y^3 +x y t=0\}.
\end{align*}
\end{thm}

We note that $R_{i}$ is the cone over a curve if and only if $i \in \{1, 2\}$.
On the other hand, there is the classification of non-normal Gorenstein del Pezzo surfaces \cite{Reid94, A-F}. 
In particular, they classified non-normal cubic surfaces as follows:

\begin{thm}[{\cite[Theorem 1.5]{A-F}}]\label{thm:nnorm2}
Let $S$ be a non-normal cubic surface with the normalization $\sigma \colon \ol S \to S$. 
Let $D \subset S$ (resp.\ $\ol D \subset \ol S$) be the conductor locus. 
Then $D$ is a line and one of the following holds:
\begin{enumerate}
\item[\textup{(C)}] $\ol S \cong \F_{1}$, $\sigma^{*}(-K_{S}) \sim \Sigma_{1}+2f_{1}$ and $\ol D \sim \Sigma_{1} +f_{1}$.
\item[\textup{(E1)}] $\ol S$ is isomorphic to the cone $S_{3} \subset \P^{4}$ over the twisted cubic in $\P^{3}$.
Moreover, $\sigma^{*}(-K_{S}) \sim \mc{O}_{\P^{4}}(1)|_{S_{3}}$ and $\ol D$ is the sum of two rulings or a non-reduced ruling of length two.
\end{enumerate}
\end{thm}

Let us check the correspondence of the notation of non-normal cubic surfaces in Theorems \ref{thm:nnorm1} and \ref{thm:nnorm2}.

\begin{lem}\label{lem:2.16}
Fix $i \in \{1,2,3,4\}$ and let $\sigma \colon \ol R_{i} \to R_{i}$ be the normalization.
Write $D \subset R_{i}$ (resp.\ $\ol D \subset \ol R_{i}$) as the conductor locus. 
Then:
\begin{enumerate}
\item[\textup{(1)}] $R_{i}$ belongs to the class \textup{(E1)} if $i \leq 2$ and the class \textup{(C)} if $i \geq 3$.
\item[\textup{(2)}] $\ol D$ is irreducible if $i=3$ and reducible if $i=4$.
\end{enumerate}
\end{lem}

\begin{proof}
(1) Since the class (E1) corresponds to the cones over curves, the assertion holds.\\
(2) Suppose that $i \in \{3, 4\}$. 
Then $R_{i}$ belongs to the class (C).
Let $L \neq D$ be a line in $R_i$ and write $L_{\ol R_{i}} \sim a\Sigma_{1} +b f_{1}$ with $a, b \in \Z$. 
Since $1=(-K_{R_i} \cdot L)=(-\sigma^*K_{R_i} \cdot L_{\ol R_{i}})= (\Sigma_{1}+2f_{1} \cdot a\Sigma_{1} +b f_{1}) =a+b$, we obtain $L_{\ol R_{i}} \sim \Sigma_{1}$ or $f_{1}$. 
Hence we have 
\begin{align}
\ol{D} \text{ is irreducible } 
&\iff \Sigma_{1} \cap \ol{D} = \emptyset 
\\
&\iff \text{There is a line in } R_{i} \text{ disjoint from } D.\nonumber
\end{align}
Now suppose that $i=3$. 
Then $\{z=y+t=0\}$ is a line in $R_{3}$ disjoint from $D=\{x=y=0\}$ and hence $\ol D$ is irreducible.
On the other hand, suppose that $i=4$.
Then $R_{4} \setminus D = R_{4} \cap \{x \neq 0\} \cong \A^{2}$ contains no proper curves and hence $\ol D$ is reducible. 
\end{proof}

\section{Curves in cubic surfaces}\label{sec:curve}

In this section, we will look closely at curves in cubic surfaces.

\subsection{Smooth rational curves in $G_{1}, G_{2}, G_{3}$ or $G_{5}$}\label{subsec:7.1}
In this subsection, we determine smooth rational curves in $G_{1}, G_{2}, G_{4}$ or $G_{5}$. 
For this, we use the following notation:
\begin{nota}\label{nota:normrat}
We assume that $S_{2}=G_{i}$ for some $i \in \{1,2,4,5\}$. 
Let $R \subset S_{2}$ be a smooth rational curve of degree $n$.
By \cite[Appendix: Configurations of the Singularity types]{Qia}, $\wt S_{2}$ is given by the successive blow-ups $\{\eta_{j} \colon S_{2,j} \to S_{2,j-1}\}_{j=1}^{6}$ with $S_{2,0} \coloneqq \P^{2}$ and $S_{2,6}=\wt S_{2}$ at points $p_{1}, \ldots, p_{6}$ which correspond to the bullets as in Figure \ref{fig:bl}.
Write $\eta \coloneqq \eta_{1} \circ \dots \circ \eta_{6} \colon \wt S_{2} \to \P^{2}$.
Set $H = \eta^{*}\mathcal{O}_{\P^2}(1)$, $E_{6} \coloneqq E_{\eta_{6}}$ and $E_{i} = (\eta_{6} \circ \dots \circ \eta_{i+1})^{*}E_{\eta_{i}}$ for $1 \leq i \leq 5$. 
Write $R_{\wt S_{2}}  \sim aH - \sum_{i=1}^{6} b_{i} E_{i}$ with $a, b_{1}, \dots, b_{6} \in \Z$.
\end{nota}

\begin{figure}[htbp]
\centering
\begin{minipage}{\textwidth}
\begin{center}
\includegraphics[clip, width=60mm]{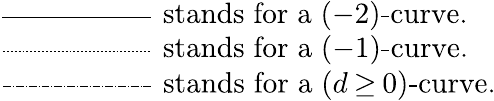}
\end{center}
\end{minipage}

\begin{minipage}{\textwidth}
\centering
\includegraphics[clip, width=120mm]{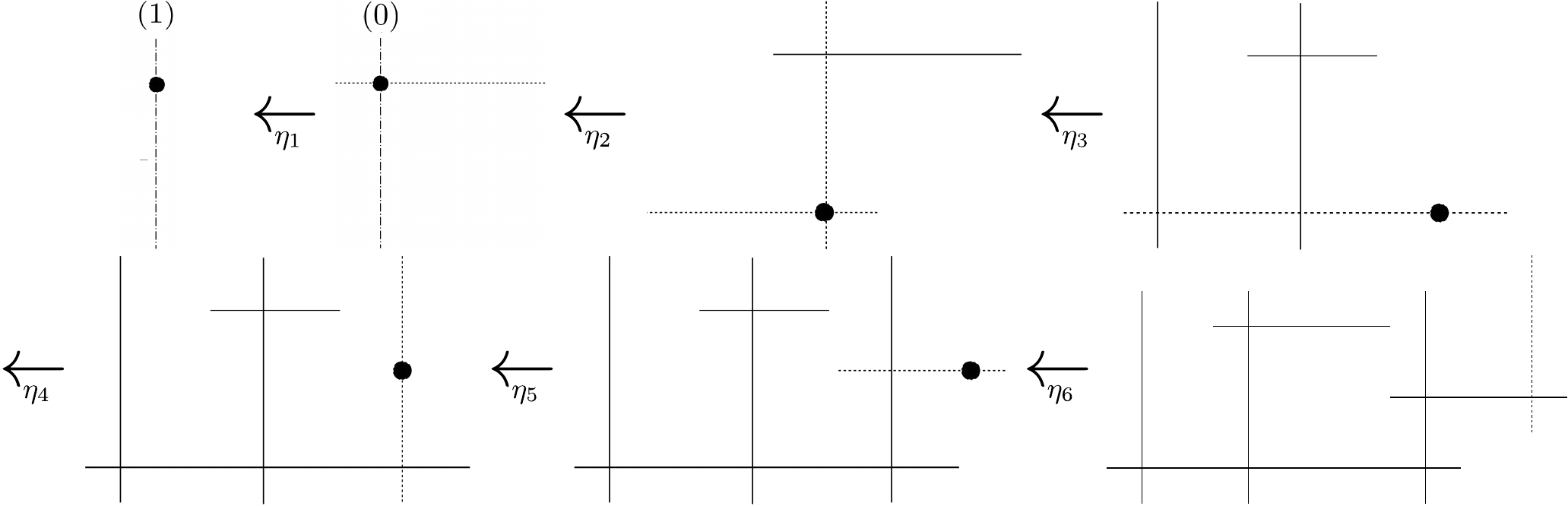}
\subcaption{The case $S_{2}=G_{1}$}
\end{minipage}

\begin{minipage}{\textwidth}
\centering
\includegraphics[clip,width=120mm]{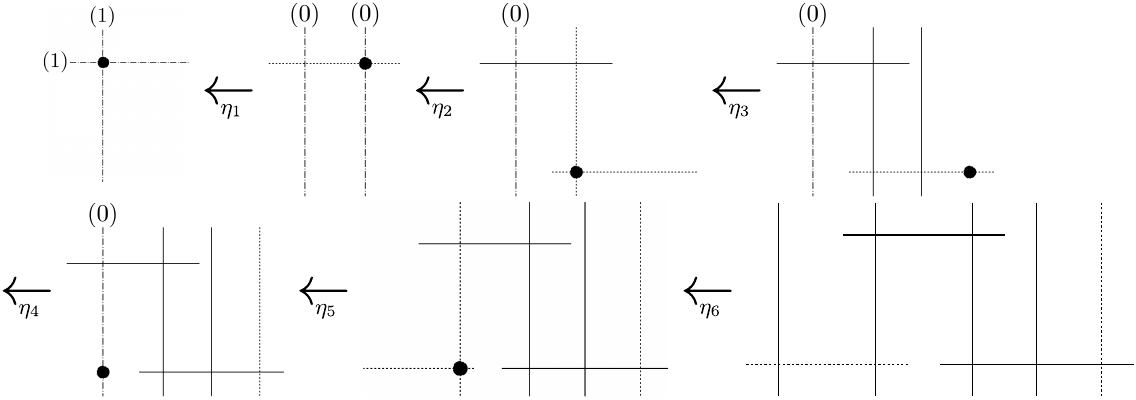}
\subcaption{The case $S_{2}=G_{2}$}
\end{minipage}

\begin{minipage}{\textwidth}
\centering
\includegraphics[clip,width=120mm]{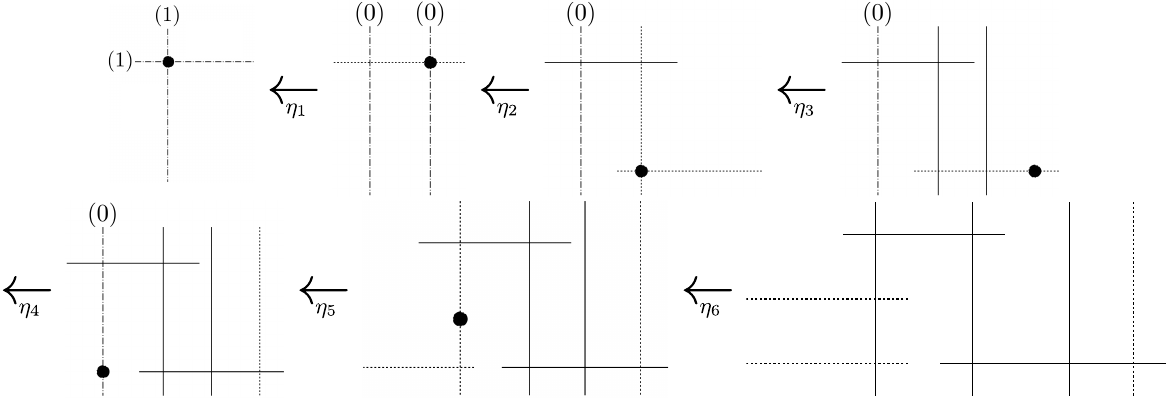}
\subcaption{The case $S_{2}=G_{4}$}
\end{minipage}

\begin{minipage}{\textwidth}
\centering
\includegraphics[clip,width=120mm]{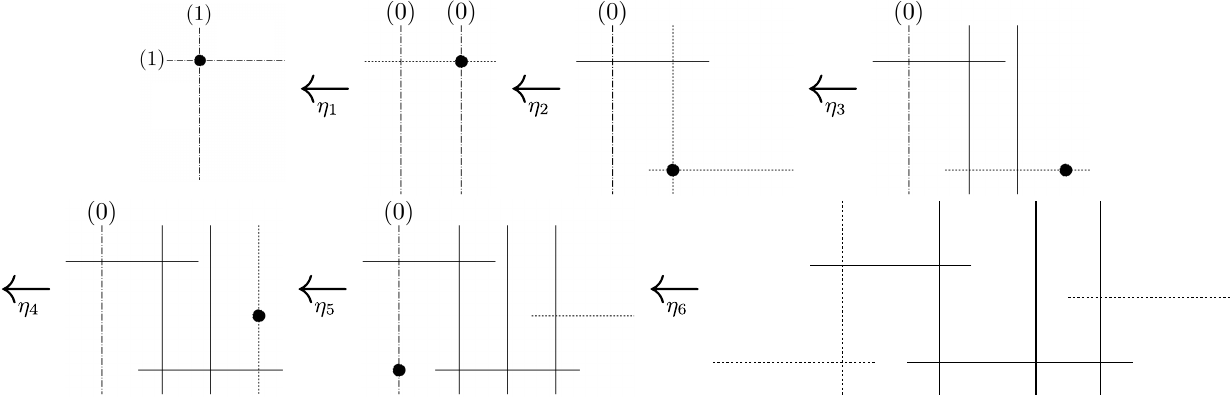}
\subcaption{The case $S_{2}=G_{5}$}
\end{minipage}
\caption{The construction of $\wt S_{2}$ when $S_{2}=G_{1}, G_{2}, G_{4}$ or $G_{5}$}
\label{fig:bl}
\end{figure}

\textbf{Notation for Figure \ref{fig:bl}}
\begin{itemize}
\item The upper left pictures represent the configurations of important lines in $\P^2$.
\item The lower right pictures represent the configurations of all the curves with negative self-intersection numbers in $\wt S_{2}$.
\end{itemize}

Table \ref{tb:negativecurve} gives all the linear equivalence classes of $(-2)$-curves in $\wt S_{2}$.
\begin{table}[htbp]
\centering
\caption{$(-2)$-curves in $\wt S_{2}$}
\label{tb:negativecurve}
\begin{tabular}{|l|l|}\hline
$i=$            &Linearly equivalence classes of $(-2)$-curves                              \\
                   &in $\wt S_{2}$ in the case where $S_{2}=G_{i}$\\ \hline
$1$                    & $H-E_{1}-E_{2}-E_{3}$, $E_{1}-E_{2}$, $E_{2}-E_{3}$,\\ 
                          &  $E_{3}-E_{4}$,\ \  $E_{4}-E_{5}$, $E_{5}-E_{6}$         \\ \hline
$2$                    & $H-E_{1}-E_{2}-E_{3}$, $E_{1}-E_{2}$, $E_{2}-E_{3}$,\\ 
                          & $H-E_{1}-E_{5}-E_{6}$, $E_{3}-E_{4}$, $E_{5}-E_{6}$      \\ \hline
$4$                    & $H-E_{1}-E_{2}-E_{3}$, $E_{1}-E_{2}$, $E_{2}-E_{3}$,\\ 
                          & $H-E_{1}-E_{5}-E_{6}$, $E_{3}-E_{4}$         \\ \hline
$5$                    & $H-E_{1}-E_{2}-E_{3}$, $E_{1}-E_{2}$, $E_{2}-E_{3}$,\\ 
                          & $E_{3}-E_{4}$, \ $E_{4}-E_{5}$                         \\ \hline
\end{tabular}
\end{table}

\begin{lem}\label{lem:3.2}
We follow Notation \ref{nota:normrat}. 
Suppose that $i=1,2$ or $4$.
Take $p$ as the singular point of $S_{2}$ when $i=1$ or $2$ and as the $A_{5}$-singularity in $S_{2}$ when $i=4$.
Then the projection $\pi \colon S_{2} \dashrightarrow \P^2$ from $p$ satisfies $\eta = \pi \circ \mu$.  
\end{lem}

\begin{proof}
We give the proof only for the case where $i=1$; the same proof works for the case where $i=2$ or $4$.

Let $N \subset \left| -K_{S_{2}}\right|$ be the net defining $\pi$. 
Then the free part of $\mu^{*} N$ defines $\pi \circ \mu$. 
Since the sum of all the $(-2)$-curves are contained in the fixed part of $\mu^{*}N$, the free part of $\mu^{*} N$ is the same as that of
\begin{align}
N_{1} \coloneqq \mu^{*}N-(H-E_{2}-E_{3}-E_{6}) \subset \left| 2H-E_{1}-E_{4}-E_{5} \right|.
\end{align}
Since 
$(2H-E_{1}-E_{4}-E_{5} \cdot E_{3}-E_{4})=-1$, the fixed part of $N_{1}$ contains the $(-2)$-curve linearly equivalent to $E_{3}-E_{4}$.
Hence the free part of $N_{1}$ is the same as that of $N_{1}-(E_{3}-E_{4})$. 
We continue in this fashion to conclude that the free part of $\mu^{*} N$ is the same as that of
\begin{align}
\mu^{*}N-(2H-E_{1}-E_{2}-E_{3}-E_{4}-E_{5}-E_{6}) \subset \left| H \right|.
\end{align}
Since $\left| H \right|$ is a free net, $\pi \circ \mu$ is defined by $\left| H \right|$. 
Hence $\eta = \pi \circ \mu$.
\end{proof}

\begin{lem}\label{lem:G1curve}
We follow Notation \ref{nota:normrat}. 
Suppose that $S_{2}=G_{1}$.
Then $R_{\wt S_{2}}$ is linearly equivalent to either $E_{6}$ or $H-E_{1}$.
\end{lem}

\begin{proof}
%
By assumption, 
\begin{align}\label{eq:ab1'}
n = (R_{\wt S_{2}} \cdot -K_{\wt S_{2}})=(aH - \sum_{i=1}^{6} b_{i} E_{i} \cdot 3H - \sum_{i=1}^{6} E_{i}) = 3a-\sum_{i=1}^{6} b_{i}.
\end{align}
Since $R$ is smooth rational curve, the genus formula yields
\begin{align}\label{eq:ab2'}
n-2 = R_{\wt S_{2}}^{2} = a^{2}-\sum_{i=1}^{6} b_{i}^{2}.
\end{align}

Now suppose that $\Sing S_{2} \not \in R$.
Then $a=n$ by Lemma \ref{lem:3.2}.
We also have $b_{1}=b_{2}=b_{3}=b_{4}=b_{5}=b_{6}$ and $a=b_{1}+b_{2}+b_{3}$ since $R_{\wt S_{2}}$ is disjoint from each $(-2)$-curve.
Then (\ref{eq:ab2'}) yields 
$
0=a^{2}-\sum_{i=1}^{6}b_{i}^{2}-n+2=3b_{1}(b_{1}-1)+2,
$
which contradicts the assumption that $b_{1} \in \Z$.

Hence $\Sing S_{2} \in R$.
Then $a=n-1$ by Lemma \ref{lem:3.2}.
We also have $\sharp(R_{\wt S_{2}} \cap E_{\mu})=1$ since $R$ is smooth.
In particular, one of the following holds:
\begin{enumerate}
\item $b_{1}=b_{2}=b_{3}=b_{4}=b_{5}=b_{6}$ and $a=b_{1}+b_{2}+b_{3}+1$.
\item $b_{1}=b_{2}=b_{3}=b_{4}+1=b_{5}+1=b_{6}+1$ and $a=b_{1}+b_{2}+b_{3}+1$. 
\item $b_{1}= \cdots = b_{j}=b_{j+1}+1= \cdots =b_{6}+1$ and $a=b_{1}+b_{2}+b_{3}$ for some $1 \leq j \leq 5$. 
\item $b_{1}= \cdots = b_{j}=b_{j+1}+1=b_{j+2}+2= \cdots =b_{6}+2$ and $a=b_{1}+b_{2}+b_{3}$ for some $1 \leq j \leq 4$.
\end{enumerate}

If the case (1) (resp.\ (2)) holds, then $n-3a+\sum_{i=1}^{6}b_{i}=-1$ (resp.\ $-4$), a contradiction with (\ref{eq:ab1'}).
Similarly, the case (4) implies $n-3a+\sum_{i=1}^{6}b_{i}=-2$ when $j=1$ or $4$, and $-4$ when $j=2$ or $3$, a contradiction.

Hence the case (3) holds.
(\ref{eq:ab1'}) now yields $j=1$ or $5$.
Suppose that $j=1$. 
Then (\ref{eq:ab2'}) yields 
$
0=a^{2}-\sum_{i=1}^{6}b_{i}^{2}-n+2=(3b_{1}-2)(b_{1}-1).
$
Since $b_{1} \in \Z$, we obtain $b_{1}=1$. 
Hence $R_{\wt S_{2}} \sim H-E_{1}$.

Suppose that $j=5$. 
Then (\ref{eq:ab2'}) yields 
$
0=a^{2}-\sum_{i=1}^{6}b_{i}^{2}-n+2=b_{1}(3b_{1}-1).
$
Since $b_{1} \in \Z$, we obtain $b_{1}=0$. 
Hence $R_{\wt S_{2}} \sim E_{6}$.
\end{proof}

\begin{lem}\label{lem:G2curve}
We follow Notation \ref{nota:normrat}. 
Suppose that $S_{2}=G_{2}$.
Then $R_{\wt S_{2}}$ is linearly equivalent to one of the following:
\begin{align*}
E_{4}, E_{6}, H-E_{1}, H-E_{5}, 2H-E_{1}-E_{2}-E_{5}.
\end{align*}
\end{lem}

\begin{proof}
As in the proof of Lemma \ref{lem:G1curve}, we obtain (\ref{eq:ab1'}) and (\ref{eq:ab2'}).
On the other hand, $E_{5}-E_{6}$ corresponds to a $(-2)$-curve and $E_{6}$ is a $(-1)$-curve by construction.

Suppose that $b_{5}<0$. Then $(R_{\wt S_{2}} \cdot (E_{5}-E_{6})+E_{6})=b_{5}<0$. 
Hence $R_{\wt S_{2}} \sim E_{6}$ and $b_{5}=0$, a contradiction.
On the other hand, if $b_{6}<0$, then $(R_{\wt S_{2}} \cdot E_{6})=b_{6}<0$, which implies $R_{\wt S_{2}} \sim E_{6}$. 
In the rest of the proof, we may assume that $b_{5} \geq 0$ and $b_{6} \geq 0$.

Write $p \in S_{2}$ as the $A_{5}$-singularity of $S_{2}$.
Suppose that $p \not \in R$.
Then $a=n$ by Lemma \ref{lem:3.2}.
We also have $b_{1}=b_{2}=b_{3}=b_{4}$ and $a=b_{1}+b_{2}+b_{3}=b_{1}+b_{5}+b_{6}$ since $R_{\wt S_{2}}$ is disjoint from $E_{\mu}$. 
(\ref{eq:ab2'}) now yields 
$
0=a^{2}-\sum_{i=1}^{6}b_{i}^{2}-n+2=(b_{1}-1)(b_{1}-2)+2b_{5}b_{6}.
$
Since $b_{1} \in \Z$, we obtain $b_{5}b_{6}=0$ and $b_{1}=1$ or $2$.
On the other hand, we have $(R_{\wt S_{2}} \cdot E_{5}-E_{6})=b_{5}-b_{6}=2(b_{1}-b_{6})$.
Since $R$ is smooth, we obtain $b_{1}=b_{6}=b_{5}$. 
Hence $b_{5}b_{6}=b_{1}^{2}>0$, a contradiction. 

Hence $p \in R$.
Then $a=n-1$ by Lemma \ref{lem:3.2}.
We also have $\sharp(R_{\wt S_{2}} \cap \mu^{-1}(p))=1$ since $R$ is smooth. 
In particular, one of the following holds:
\begin{enumerate}
\item $b_{1}=b_{2}=b_{3}=b_{4}$ and $a=b_{1}+b_{2}+b_{3}+1=b_{1}+b_{5}+b_{6}$.
\item $b_{1}=b_{2}=b_{3}=b_{4}+1$ and $a=b_{1}+b_{2}+b_{3}+1=b_{1}+b_{5}+b_{6}$. 
\item $b_{1}=b_{2}=b_{3}=b_{4}$ and $a=b_{1}+b_{2}+b_{3}=b_{1}+b_{5}+b_{6}+1$.
\item $b_{1}=b_{2}+1=b_{3}+1=b_{4}+1$ and $a=b_{1}+b_{2}+b_{3}=b_{1}+b_{5}+b_{6}+1$. 
\item $b_{1}= \cdots = b_{j}=b_{j+1}+1= \cdots =b_{4}+1$ and $a=b_{1}+b_{2}+b_{3}=b_{1}+b_{5}+b_{6}$ for some $1 \leq j \leq 3$. 
\item $b_{1}= \cdots = b_{j}=b_{j+1}+1=b_{j+2}+2= \cdots =b_{4}+2$ and $a=b_{1}+b_{2}+b_{3}=b_{1}+b_{5}+b_{6}$ for some $1 \leq j \leq 2$.
\end{enumerate}

If one of the cases (2), (4) and (6) holds, then $n-3a+\sum_{i=1}^{6}b_{i}=-1$, a contradiction with (\ref{eq:ab1'}).

Suppose that the case (3) holds.
Then (\ref{eq:ab2'}) yields 
$
0=a^{2}-\sum_{i=1}^{6}b_{i}^{2}-n+2=b_{1}(b_{1}+1)+2b_{5}b_{6}.
$
Since $b_{1} \in \Z$, we obtain $b_{5}b_{6}=0$ and $b_{1}=-1$ or $0$.
In each case, we have $b_{5}+b_{6}=2b_{1}-1<0$, which contradicts the assumption that $b_{5} \geq 0$ and $b_{6} \geq 0$.
Hence one of the cases (1) and (5) holds.

Suppose that the case (1) holds.
Then (\ref{eq:ab2'}) yields 
$
0=a^{2}-\sum_{i=1}^{6}b_{i}^{2}-n+2=b_{1}(b_{1}-1)+2b_{5}b_{6}.
$
Since $b_{1} \in \Z$, we obtain $b_{5}b_{6}=0$ and $b_{1}=0$ or $1$.
On the other hand, we have $(R_{\wt S_{2}} \cdot E_{5}-E_{6})=b_{5}-b_{6}=2(b_{1}-b_{6})+1$.
Since $R$ is smooth, we obtain $b_{1}=b_{6}=b_{5}-1$.
Since $b_{5}b_{6}=b_{1}(b_{1}+1)=0$, we conclude that $b_{1}=0$ and $R_{\wt S_{2}} \sim H-E_{5}$.

Suppose that the case (5) holds.
Then (\ref{eq:ab2'}) yields 
$
0=a^{2}-\sum_{i=1}^{6}b_{i}^{2}-n+2=b_{1}(b_{1}-1)+2b_{5}b_{6}.
$
Since $b_{1} \in \Z$, we obtain $b_{5}b_{6}=0$ and $b_{1}=0$ or $1$.

Suppose that $j=1$ in addition.
Then $a=3b_{1}-2 \geq 0$ by the nefness of $H$.
Hence $b_{1}=1$ and $R_{\wt S_{2}} \sim H-E_{1}$.
Similarly, if $j=2$, then $b_{1}=1$ and $R_{\wt S_{2}} \sim 2H-E_{1}-E_{2}-E_{5}$.

Suppose that $j=3$ in addition.
Then 
$(R_{\wt S_{2}} \cdot E_{5}-E_{6})=b_{5}-b_{6}=2(b_{1}-b_{6})$. 
Since $R$ is smooth, we obtain $b_{1}=b_{5}=b_{6}$.
Since $b_{5}b_{6}=b_{1}^{2}=0$, we conclude that $b_{1}=0$ and $R_{\wt S_{2}} \sim E_{4}$.

Combining these results, we complete the proof.
\end{proof}

\begin{lem}\label{lem:G4curve}
We follow Notation \ref{nota:normrat}. 
Suppose that $S_{2}=G_{4}$.
Then $R_{\wt S_{2}}$ is linearly equivalent to one of the following for some $i \in \{5, 6\}$:
\begin{align*}
&E_{4}, E_{i}, H-E_{1}, H-E_{i}, 2H-E_{1}-E_{2}-E_{i},\\ 
&3H-E_{1}-E_{2}-E_{3}-E_{4}-2E_{i}, 3H-E_{1}-E_{2}-E_{3}-2E_{i}, \\
&4H-E_{1}-E_{2}-E_{3}-E_{4}-3E_{i}, 6H-2E_{1}-2E_{2}-2E_{3}-2E_{4}-4E_{i}.
\end{align*}
\end{lem}

\begin{proof}
As in the proof of Lemma \ref{lem:G1curve}, we obtain (\ref{eq:ab1'}) and (\ref{eq:ab2'}).
On the other hand, $E_{5}$ and $E_{6}$ are $(-1)$-curves by construction.

If $b_{i}<0$ for some $i \in \{5, 6\}$, then $(R_{\wt S_{2}} \cdot E_{i})=b_{i}<0$ and hence $R_{\wt S_{2}} \sim E_{i}$.
In the rest of the proof, we may assume that $b_{5} \geq 0$ and $b_{6} \geq 0$.

Suppose that $\Sing S_{2} \not \in R$.
Then $a=n$ by Lemma \ref{lem:3.2}.
We have $b_{1}=b_{2}=b_{3}=b_{4}$ and $a=b_{1}+b_{2}+b_{3}=b_{1}+b_{5}+b_{6}$ since $R_{\wt S_{2}}$ is disjoint from $E_{\mu}$. 
As in the proof of Lemma \ref{lem:G2curve}, we also obtain $b_{5}b_{6}=0$ and $b_{1}=1$ or $2$.
Hence $R_{\wt S_{2}} \sim 2H-E_{1}-E_{2}-E_{3}-E_{4}-2E_{i}$ or $4H-2E_{1}-2E_{2}-2E_{3}-2E_{4}-4E_{i}$ for some $i \in \{5, 6\}$.

In the remainder of the proof, we assume that $\Sing S_{2} \in R$.
As in the proof of Lemma \ref{lem:G2curve}, one of the cases (1) and (5) holds.

Suppose that the case (1) holds.
Then $b_{5}b_{6}=0$ and $b_{1}=0$ or $1$, which implies $R_{\wt S_{2}} \sim H-E_{i}$ or $4H-E_{1}-E_{2}-E_{3}-E_{4}-3E_{i}$ for some $i \in \{5, 6\}$.

Suppose that the case (5) holds.
Then $b_{5}b_{6}=0$ and $b_{1}=0$ or $1$.
Suppose that the former holds in addition. 
Then $a=j-3$. 
By the nefness of $H$, we obtain $j=3$ and $R_{\wt S_{2}} \sim E_{4}$.
If the latter holds, we obtain $R_{\wt S_{2}} \sim H-E_{1}$ when $j=1$, $2H-E_{1}-E_{2}-E_{i}$ when $j=2$ and $3H-E_{1}-E_{2}-E_{3}-2E_{i}$ when $j=3$ for some $i \in \{5, 6\}$.

Combining these results, we complete the proof.
\end{proof}

\begin{lem}\label{lem:G5curve}
We follow Notation \ref{nota:normrat}. 
Suppose that $S_{2}=G_{5}$ and $n \leq 2$.
Then $R_{\wt S_{2}}$ is linearly equivalent to one of the following:
\begin{align*}
&E_{5}, E_{6}, H-E_{1}-E_{6}, H-E_{1}, H-E_{6}, \\
&3H-E_{1}-E_{2}-E_{3}-E_{4}-E_{5}-2E_{6}.
\end{align*} 
\end{lem}

\begin{proof}
As in the proof of Lemma \ref{lem:G1curve}, we obtain (\ref{eq:ab1'}) and (\ref{eq:ab2'}).
The Cauchy-Schwarz inequality now shows that  
\begin{align}
a^2-(n-2) = \sum_{i=1}^6 b_{i}^2 \geq 6 \left(\frac{\sum_{i=1}^6 b_{i}}{6}\right)^2 =\frac{(3a-n)^2}{6}
\end{align}
and hence $2(n^{2}-3n+6) \geq 3(a-n)^{2}$.
Combining (\ref{eq:ab1'}), (\ref{eq:ab2'}) and this inequality, we conclude that 
the $8$-tuple $(n, a; b_{1}, \ldots, b_{6})$ is one of the following up to the action of  the symmetric group $\mathfrak{S}_{6}$ on the second factor of $\Z^{2} \oplus \Z^{6}$ as a permutation:
\begin{align}\label{eq:8tuples}
&(1,0; -1,0^{5}), &&(1,1; 1^{2},0^{4}), &&(1,2; 1^{5},0), \\
&(2,1; 1,0^{5}), &&(2, 2; 1^{4}, 0^{2}),&&(2, 3; 2, 1^{5}).\nonumber
\end{align}
On the other hand, we have $0 \leq b_{i}-b_{i+1} \leq 1$ for each $i \in \{1,2,3,4\}$ and $0 \leq a-b_{1}-b_{2}-b_{3} \leq 1$ since $R_{\wt S_{2}}$ intersects with each $(-2)$-curve with multiplicity at most one.
Combining (\ref{eq:8tuples}) and these inequalities, we get the assertion.
\end{proof}

\begin{lem}\label{lem:G1cusp}
We follow Notation \ref{nota:normrat}. 
Suppose that $S_{2}=G_{1}$.
If $R' \subset S_{2}$ is a cuspidal cubic, 
Then $R'_{\wt S_{2}} \sim H$.
\end{lem}

\begin{proof}
Since $R'$ is a Cartier divisor, we have $\Sing R'=\Sing S_{2}$.
Hence $R'_{\wt S_{2}}$ is a smooth rational curve.
Write $R'_{\wt S_{2}} = aH-\sum_{i=1}^{6}b_{i}E_{i}$ with $a, b_{1}, \ldots, b_{6} \in \Z$.
An analysis similar to that in the proof of Lemma \ref{lem:G5curve} shows that
$(n, a;$ $b_{1}$, $\ldots$, $b_{6})$ is one of the following up the $\mathfrak{S}_{6}$-action:
\begin{align}\label{eq:8tuples'}
&(3, 1; 0^{6}), &&(3, 2;1^{3},0^{3}), &&(3, 3; 2, 1^{4}, 0), &&(3, 4;2^{3},1^{3}), &&(3, 5;2^{6}).
\end{align}
On the other hand, we have $0 \leq b_{i}-b_{i+1} $ for each $i \in \{1,2,3,4, 5\}$ and $0 \leq a-b_{1}-b_{2}-b_{3}$ since $R'_{\wt S_{2}}$ is distinct from any $(-2)$-curves in $\wt S_{2}$.
Combining (\ref{eq:8tuples'}) and these inequalities, we get the assertion.
\end{proof}

\subsection{Smooth curves in the cones over elliptic curves}\label{subsec:7.2}
Next we investigate smooth curves in the cones over an elliptic curves.

\begin{lem}\label{lem:3.5}
Suppose that $S_{2}$ is the cone over an elliptic curve.
Let $R$ be a smooth curve in $S_{2}$.
Then $\deg R \neq 2$.
Moreover, $p_{a}(R)=0$ when $\deg R=1$, $p_{a}(R)=1$ when $\deg R=3$ or $4$, and $p_{a}(R)\geq 4$ when $\deg R  \geq 5$.
\end{lem}

\begin{proof}
Take $C_{0}$ and $f$ as the minimal section and a fiber of $\P^{1}$-bundle structure on $\wt S_{2}$ respectively.
Write $R_{\wt S_{2}} \equiv a C_0 +bf$ with $a, b \in \Z_{\geq 0}$.
Since $R$ is smooth, we have $b-3a=(R_{\wt S_{2}} \cdot C_{0})=0$ or $1$. 
The genus formula now yields
\begin{align}
p_{a}(R_{\wt S_{2}}) = 
\begin{cases}
1+ \frac32 a(a-1)& \text{ if } b=3a>0,\\
1 + \frac12 (3a+2)(a-1)& \text{ if } b=3a+1.
\end{cases}
\end{align}
Hence $g_{a}(R)=0$ when $a=0$, $g_{a}(R)=1$ when $a=1$ and $g_{a}(R) \geq 4$ when $a \geq 2$.
We also have $\deg R =(-K_{S_{2}} \cdot R) = (C_0 +3f \cdot a C_0 +bf)=b$. 
Combining these results, we obtain the assertion.
\end{proof}

\subsection{Curves in non-normal cubic surfaces}\label{subsec:7.3}
Finally we investigate curves in non-normal cubic surfaces.

\begin{lem}\label{lem:3.6}
Suppose that $S_{2} = R_{1}$ or $R_{2}$. 
Let $R \subset S_{2}$ be a smooth curve distinct from $E$.
Then $\deg R \neq 2$.
Moreover, $p_{a}(R)=0$ when $\deg R \leq 4$, and $p_{a}(R) \geq 2$ when $\deg R \geq 5$.
\end{lem}

\begin{proof}
By Lemma \ref{lem:2.16}, $S_{2}$ belongs to the class (E1) as in Theorem \ref{thm:nnorm2}. 
Write $R_{\wt S_{2}} \sim a \Sigma_{3} +bf_{3}$ with $a, b \in \Z_{\geq 0}$.
Since $R$ is smooth, we have $b-3a=(R_{\wt S_{2}} \cdot \Sigma_{3})=0$ or $1$. 
The genus formula now yields
\begin{align}
p_{a}(R_{\wt S_{2}}) = 
\begin{cases}
1+ \frac12 a(3a-5)& \text{ if } b=3a>0,\\
\frac32 a(a-1)& \text{ if } b=3a+1.
\end{cases}
\end{align}
Hence $g_{a}(R)=0$ when $a \leq 1$ and $g_{a}(R) \geq 2$ when $a \geq 2$.
We also have $\deg R =(R_{\wt S_{2}} \cdot \mu^*(-K_{S_{2}})) =(a \Sigma_{3} +bf_{3} \cdot  \Sigma_{3} +3f_{3})=b$. 
Combining these results, we obtain the assertion.
\end{proof}

\begin{lem}\label{lem:3.7}
Suppose that $S_{2} = R_{3}$ or $R_{4}$.
Let $R \subset S_{2}$ be a curve distinct from $E$.
Write $R_{\wt S_{2}} \sim a \Sigma_{1} +bf_{1}$ with $a, b \in \Z_{\geq 0}$.
Then $\deg R =a+b$.
Moreover, each line in $R_{4}$ is the image of a fiber of the $\P^{1}$-bundle structure on $\wt S_{2} \cong \F_{1}$.
\end{lem}

\begin{proof}
By Lemma \ref{lem:2.16}, $S_{2}$ belongs to the class (C) as in Theorem \ref{thm:nnorm2}. 
Hence $\deg R =(R_{\wt S_{2}} \cdot \mu^*(-K_{S_{2}})) =(a \Sigma_{1} +bf_{1} \cdot \Sigma_{1} +2f_{1})=a+b$, which is the first assertion.
In particular, $R_{\wt S_{2}} \sim \Sigma_{1}$ or $f_{1}$ when $\deg R =1$.
If $S_{2}=R_{4}$, Lemma \ref{lem:2.16} shows that there is a curve $L \sim f_{1}$ such that $\mu(L)=\mu(\Sigma_{1})$.
Hence the second assertion holds.
\end{proof}

\section{The case $p_{a}(C) \geq 1$}\label{sec:ell}
In the rest of the paper, we prove Theorems \ref{thm:main3tuple} and \ref{thm:mainisom}.
In this section, we treat the case where $p_{a}(C) \geq 1$.

\begin{lem}\label{lem:4.2}
If $U$ is an affine homology $3$-cell, then $p_{a}(C) \leq 1$.
\end{lem}

\begin{proof}
Combining Lemma \ref{lem:2.10} and \cite[Theorem 2.1]{Nag1}, we obtain
\begin{align}\label{eq:ell-1}
\eu(F) 
&= \eu(S_{2}) +2p_{a}(C)+ N_{1} +N_{2} - N_{1 \cap 2}-2\\ 
&\geq 1+2p_{a}(C)+1-2=2p_{a}(C).\nonumber
\end{align}
On the other hand, we have $4 \geq \eu(F)$ since $F \subset S_{1}$ is a plane cubic. Hence $2 \geq p_{a}(C)$ and it suffices to show that $p_{a}(C) \neq 2$.

Conversely, suppose that $p_{a}(C)=2$. 
Then $(\ref{eq:ell-1})$ shows that $\eu(S_{2})=1$ and hence $S_{2}$ is the cone over an elliptic curve by \cite[Theorem 2.1]{Nag1}. 
However, Lemma \ref{lem:3.5} shows that $S_{2}$ contains no smooth curve with $p_{a}=2$, a contradiction.
\end{proof}


\begin{lem}\label{lem:4.3}
If $U$ is an affine homology $3$-cell and $p_{a}(C) = 1$, then $S_{2}$ is the cone over an elliptic curve. 
\end{lem}

\begin{proof}
As in the proof of Lemma \ref{lem:4.2}, we obtain
\begin{equation}\label{eq:301}
4 \geq \eu(F) = \eu(S_{2}) + N_{1} +N_{2} - N_{1 \cap 2} \geq 1 + \eu (S_{2}).
\end{equation}

Suppose that $S_{2}$ is non-normal. 
Since $S_{2}$ contains a smooth curve with $p_{a}=1$, Lemma \ref{lem:3.6} shows that $S_{2}$ belongs to the class (C).
Then (\ref{eq:301}) shows that $3 \geq \eu(F)-1 \geq \eu(S_{2})=\eu(\F_1)-\eu(\wt{E})+\eu(E)=6-\eu(\wt{E})$. 
Hence $\wt{E}$ is reducible and $F$ is the sum of three lines intersecting in one point.
However, $S_{2}$ has no such lines by Lemma \ref{lem:3.7}, a contradiction.

Hence $S_{2}$ is normal.
It suffices to exclude the case where $S_{2}$ is rational.
Conversely, suppose that $S_{2}$ is rational. 
Then $B_2(S_{2})=1$ and $\eu(F)=4$ by (\ref{eq:301}).
Hence $S_{2}$ is projectively equivalent to $G_{i}$ for some $i \in \{1,2,3\}$ by Corollary \ref{cor:B-L} and $F$ is the sum of three lines intersecting in one point.
Table \ref{tb:line} shows, however, that $S_{2}$ has no such lines, a contradiction.
\end{proof}

\begin{prop}\label{prop:highgenus}
If $U$ is an affine homology $3$-cell and $p_{a}(C) \geq 1$, then the case \textup{(a)} of Theorem \ref{thm:main3tuple} holds.
\end{prop}

\begin{proof}
By Lemmas \ref{lem:4.2} and \ref{lem:4.3}, we have $p_{a}(C)=1$ and $S_{2}$ is the cone over an elliptic curve. 
Lemma \ref{lem:3.5} now shows that $\deg C=3$ or $4$. 
On the other hand, (\ref{eq:301}) shows that $\eu(F) = 1 + N_{1} +N_{2} - N_{1 \cap 2} \geq 2$.
Hence $F$ is a sum of rulings.
In particular, $S_{1}$ contains the vertex of $S_{2}$ and hence $N_{2}=N_{1 \cap 2}$.
Therefore $B_{2}(F)=N_{1}=\sharp(C \cap S_{1})$, and the assertion holds.
\end{proof}

Next let us check that $U \cong \A^{3}$ if the case \textup{(a)} of Theorem \ref{thm:main3tuple} holds.  

\begin{prop}\label{prop:elldeg3}
Suppose that the case \textup{(a)} of Theorem \ref{thm:main3tuple} holds and $\deg C =3$. 
Then $U \cong \A^{3}$.
\end{prop}

\begin{proof}
We can change the coordinate $\{x, y, z, t\}$ of $\P^{3}$ such that $S_{1}=\{x=0\}, S_{2} = \{f(x, y, z) = 0\}$ and $C=\{t=0\}|_{S_{2}}$ for some cubic form $f$.  
Hence $U \cong \{wf(1,y,z)+t=0\} \subset \A^4_{(y,z,t,w)}$. Therefore $U \cong \A^{3}$.
\end{proof}

\begin{prop}\label{prop:elldeg4}
Suppose that the case \textup{(a)} of Theorem \ref{thm:main3tuple} holds and $\deg C =4$. 
Then $U \cong \A^{3}$.
\end{prop}

\begin{proof}
Since $C$ is the complete intersection of two quadrics, there is a quadric fibration structure $\psi \colon V \to \P^{1}$ defined by $|\vp^{*}\mc{O}_{\P^{3}}(2) -E_{\vp}|$.

\noindent{\ul{Step 1}}: 
Let us show that $U$ is a contractible affine $3$-fold.
Let $R$ be a smooth member of $|\mc{O}_{\P^{3}}(1)|_{S_{2}}|$ and $R^{0} \coloneqq R \setminus (R \cap S_{1})$.
Then $S_{2}^{0} \coloneqq S_{2} \setminus F$ is an $\A^{1}$-bundle over $R^{0}$.
Since $\sharp(C \cap S_{1})=B_{2}(F)$, a curve $C^{0} \coloneqq C \setminus (C \cap S_{1}) \subset S_{2} \setminus F$ is a section of this $\A^{1}$-bundle.
\cite[Corollary 3.1]{K-Z} now shows that $U$ is a contractible affine $3$-fold because $U$ is the affine modification of $\A^{3}$ with the locus $(C^{0} \subset S_{2}^{0})$.

\noindent{\ul{Step 2}}: 
Let us show that each fiber of $\psi|_{V \setminus D_{2}}$ is isomorphic to $\A^{2}$.
Take $Q$ as a $\psi$-fiber.
Since $S_{2}$ contains no conic, there are lines $l_{1}$ and $l_{2}$ in $Q_{\P^{3}}$ (maybe $l_{1}=l_{2}$) such that $S_{2}|_{Q_{\P^{3}}}=C+l_{1}+l_{2}$ and hence $D_{2}|_{Q}=l_{1}+l_{2}$.

Suppose that $Q \cong \P^{1} \times \P^{1}$. 
Then $l_{1}$ and $l_{2}$ are distinct from each other since $l_{1}+l_{2}$ is a $(1, 1)$-divisor.
Hence $Q \cap (V \setminus D_{2})=\A^{2}$. 

In particular, $D_{2}$ is non-normal by the generic smoothness.
Since $2D_{2} \sim -K_{V}+Q$, \cite[Lemma 2.7]{Nag2} now shows that $\Sing D_{2}$ forms a $\psi$-section.

Suppose that $Q \cong \Q^{2}_{0}$. 
Then $D_{2}|_{Q}$ is singular at $\Sing D_{2} \cap Q$, which is distinct from the vertex of $Q$.
Hence $l_{1}=l_{2}$ and $Q \cap (V \setminus D_{2})=\A^{2}$.

\noindent{\ul{Step 3}}: 
Let us show that $\psi|_U$ is an $\A^{1} \times \C^*$-fibration whose special fibers are all isomorphic to $\A^{2}$. 
For this, it suffices to show that each $\psi$-fiber $Q$ satisfies $Q \cap D_{1} \cap (V \setminus D_{2}) \cong \A^{1}$ or $\emptyset$.

We note that $\Sing D_{2}$ is the same as the $\vp$-exceptional curve over the vertex of $S_{2}$ by the choice of $C$.
Hence $\Sing D_{2} \subset D_{1}$.
Thus $D_{1}|_{Q}$ is a member of $|-(1/2)K_{Q}|$ containing $\Sing D_{2} \cap Q$, and $D_{2}|_{Q}$ is the unique member of $|-(1/2)K_{Q}|$ singular at $\Sing D_{2} \cap Q$.
Hence $Q \cap D_{1} \cap (V \setminus D_{2}) = \A^{1}$ when $D_{1}|_{Q} \neq D_{2}|_{Q}$ and $\emptyset$ when $D_{1}|_{Q} = D_{2}|_{Q}$.

\noindent{\ul{Step 4}}: 
Finally let us show the assertion.
Let $m$ be the number of special fibers of $\psi|_U$.
By the contractibility of $U$, we have 
\begin{align}
1= \eu(U)=(\eu(\P^1)-m)\eu(\A^{1} \times \C^*)+m \eu(\A^{2})=m.
\end{align}
Take $Q$ as the unique $\psi$-fiber such that $Q \cap {U} \cong \A^{2}$.

On the other hand, since $D_{2}$ is non-normal, \cite[Theorem 4.2]{Nag2} shows that $U' \coloneqq V \setminus (D_{2} \cup Q) \cong \A^{3}$.
Then $\psi|_{U'} \colon U' \to \A^{1}$ is an $\A^{2}$-bundle by \cite[Main Theorem]{Kal}.
Since $D_{1} \cap U'$ is a sub $\A^{1}$-bundle of $\psi|_{U'}$, 
we obtain $U \setminus (U \cap Q)=U' \setminus (U' \cap D_{1}) \cong \A^{2} \times \C^*$ by \cite[Theorem B]{B-D}. 
\cite[Main Theorem]{Kal} now yields $U \cong \A^{3}$ because $U$ is a contractible affine $3$-fold and $U \cap Q \cong \A^{2}$.
\end{proof}

Combining Propositions \ref{prop:highgenus}--\ref{prop:elldeg4}, we complete the proof of Theorems \ref{thm:main3tuple} and \ref{thm:mainisom} when $p_{a}(C) \geq 1$.

\section{The case $p_{a}(C)=0$ and $S_{2}$ is normal}\label{sec:normal}
In the subsequent sections \S \ref{sec:normal}--\ref{sec:cont}, we prove Theorem \ref{thm:main3tuple} when $p_{a}(C)=0$.
In this section, 
we assume the following:
\begin{assu}
$p_{a}(C)=0$, $U$ is an affine homology $3$-cell and $S_{2}$ is normal.
\end{assu}

Under this assumption, we show that $(S_{1}, S_{2})$ is projectively equivalent to one of the pairs listed in the case \textup{(d)} or \textup{(e)} of Theorem \ref{thm:main3tuple}.
Firstly let us check the rationality and the singularity of $S_{2}$.

\begin{lem}\label{lem:5.1}
$S_{2}$ is rational. 
\end{lem}

\begin{proof}
We note that \cite[Lemma 2.4]{Nag1} still holds if it is just assumed that the affine modification $N$ is an affine homology $3$-cell. 
Hence it leads to
\begin{align}\label{eq:P1irrational}
H_{i}(S_{2} \setminus (F \cup \Sing S_{2}), \Z)=
\begin{cases}
\Z & i=0, \\
\Z^{\sharp (C \cap(\Sing S_{2} \cup F))-1}&i=1,\\
\Z^{\sharp(\Sing (S_{2} \setminus F))}&i=3,\\
0 & \text{otherwise}
\end{cases}
\end{align}
by setting 
$X= \P^3 \setminus S_{1} \cong \A^{3}, S= S_{2} \setminus F$ and $r=C \setminus (C \cap S_{1})$.

By \cite[Theorem 2.2]{H-W}, we only have to exclude the case where $S_{2}$ is the cone over an elliptic curve. 
Conversely, suppose that $S_{2}$ is the cone over an elliptic curve.
Then $C$ is a ruling of $S_{2}$ by Lemma \ref{lem:3.5}.

If $F$ is smooth, then $H_{2}(S_{2} \setminus (F \cup \Sing S_{2}), \Z) \cong \Z^{3}$ since $S_{2} \setminus (F \cup \Sing S_{2})$ is homeomorphic to $\C^* \times F$, a contradiction with (\ref{eq:P1irrational}).
Hence $F$ is the sum of rulings.
Then $H_{1}(S_{2} \setminus (F \cup \Sing S_{2}), \Z) \cong \Z^{B_{2}(F)+1}$ since $S_{2} \setminus (F \cup \Sing S_{2})=S_{2} \setminus F$ is homeomorphic to $\A^{1} \times (R \setminus (F \cap R))$ for some smooth member $R \in |-K_{S_{2}}|$.
On the other hand, both $C \cap F$ and $C \cap \Sing S_{2}$ coincides with the vertex of $S_{2}$.
(\ref{eq:P1irrational}) now yields $B_{2}(F)+1=0$, a contradiction. 
Hence we have the assertion.
\end{proof} 

\begin{lem}\label{lem:5.2}
$S_{2} \setminus F$ is smooth.
\end{lem}

\begin{proof}
Set $S \coloneqq S_{2} \setminus F$ and $\Sing S \coloneqq \{p_1, \dots, p_n\}$ $(n \geq 0)$.
For $1 \leq i \leq n$, $p_{i} \in S$ is a DuVal singularity and hence we can take an analytically open neighborhood $U_{i}$ of $p_{i}$ homeomorphic to the quotient of $\C^{2}$ by a finite group action.
In particular $U_{i}$ is contractible by \cite{KPTR89}.
We may also assume that $U_{i} \cap U_{j} =\emptyset$ whenever $i \neq j$.
Then the Mayer-Vietoris exact sequence for the open covering $\{ S \setminus \Sing S, \bigcup_{i=1}^{n} U_{i}\}$ gives the following exact sequence of homologies:
\begin{align}\label{diag:smooth}
\xymatrix@C=15pt@R=15pt{
&\cdots \ar[r] 
&{H_{2}(S \setminus \Sing S, \Z)}  \ar [r]
&H_{2}(S, \Z) 
\\
\ar [r] 
&{\bigoplus_{i=1}^{n}} {H_{1}(U_{i} \setminus \{p_i\}, \Z)} 
\ar [r]
&{H_{1}(S \setminus \Sing S, \Z)}  \ar [r]
& {\cdots}. 
}
\end{align}
As in the proof of Lemma \ref{lem:5.1}, \cite[Lemma 2.4]{Nag1} yields (\ref{eq:P1irrational}).
On the other hand, $H_{2}(S, \Z)$ is a free $\Z$-module since $S$ is an affine surface.
(\ref{diag:smooth}) now shows that, for each $i \geq 1$, $H_{1}(U_{i} \setminus \{p_i\}, \Z)$ is also a free $\Z$-module.
Hence $p_{i} \in S$ is the $E_8$-singularity for each $i \geq 1$ by \cite[Satz 2.8]{Brieskorn} and \cite[Theorem 1.4 (a)]{Brenton-Drucker}. 
Since 
$9 = \eu(\wt S_{2}) \geq \eu(S_{2})+8n \geq 3+8n,$
we obtain $n=0$ as desired.
\end{proof}

Next we discuss the intersection $F$.

\begin{lem}\label{lem:5.3}
It holds that $B_1(F)=0$ and $B_2(F)=B_2(S_{2}) + \sharp (C \cap F) -1$. 
Moreover, $B_2(S_{2}) \leq 3$.
\end{lem}

\begin{proof}
First let us consider the following exact sequence of cohomologies:
\begin{align}\label{eq:FBetti1}
\xymatrix@C=15pt@R=5pt{
&\cdots \ar [r]&H^{1}(\wt S_{2}, \Z) \ar [r]
&H^{1}(F_{\wt S_{2}} \cup E_{\mu}, \Z)\\
\ar[r]
&H^{2}(\wt S_{2}, F_{\wt S_{2}} \cup E_{\mu}, \Z) \ar [r]
&H^{2}(\wt S_{2}, \Z) \ar [r]
&H^{2}(F_{\wt S_{2}} \cup E_{\mu}, \Z)\\
\ar[r]
&H^{3}(\wt S_{2}, F_{\wt S_{2}} \cup E_{\mu}, \Z) \ar [r]
&H^{3}(\wt S_{2}, \Z) \ar[r]
&\cdots.
}
\end{align}
Combining the Lefschetz duality, Lemma \ref{lem:5.2} and (\ref{eq:P1irrational}), we obtain 
\begin{align}\label{eq:FBetti2}
 H^{i}(\wt S_{2}, F_{\wt S_{2}} \cup E_{\mu}, \Z) 
\cong & H_{4-i}(\wt S_{2} \setminus (F_{\wt S_{2}} \cup E_{\mu})), \Z)\\
\cong & H_{4-i} (S_{2} \setminus F, \Z)\nonumber\\
\cong &
\begin{cases}
\Z &i=4\\
\Z^{\sharp (C \cap F)-1}&i=3\\
0& \mathrm{otherwise}.
\end{cases}
\nonumber
\end{align}
We also have $H^{1}(\wt S_{2}, \Z)=H^{3}(\wt S_{2}, \Z)=0$ since $S_{2}$ is rational. 
Combining these results, we conclude that 
\begin{align}\label{eq:FBetti4}
&B_{1}(F_{\wt S_{2}} \cup E_{\mu}) = 0,\ \ B_2(F_{\wt S_{2}} \cup E_{\mu}) = B_2(\wt S_{2}) + \sharp (C \cap F) -1.
\end{align} 

Next let us consider the following exact sequence of cohomologies given by \cite[Lemma 2.5]{Nag1}:
\begin{align}\label{eq:FBetti5}
\xymatrix@C=15pt@R=5pt{
0 \ar [r]
&H^{1}(S_{2}, \Z) \ar [r]
&H^{1}(\Sing S_{2}, \Z) \oplus H^{1}(\wt S_{2}, \Z) \ar[r] 
&H^{1}(E_{\mu}, \Z)
\\
\ar [r]
&H^{2}(S_{2}, \Z) \ar [r]
&H^{2}(\Sing S_{2}, \Z) \oplus H^{2}(\wt S_{2}, \Z) \ar[r] 
&H^{2}(E_{\mu}, \Z) 
\\
\ar[r]
&H^{3}(S_{2}, \Z) \ar[r]
&H^{3}(\Sing S_{2}, \Z) \oplus H^{3}(\wt S_{2}, \Z).
&
}
\end{align}
We note that the homomorphism $H^{2}(\wt S_{2}, \Z) \to H^{2}(E_{\mu}, \Z)$ has finite cokernel by the negative definiteness of $E_{\mu}$. Hence
\begin{align}\label{eq:FBetti6}
B_{1}(S_{2})=B_{3}(S_{2})=0.
\end{align}

\cite[Lemma 2.5]{Nag1} gives another exact sequence as follows:
\begin{align}\label{eq:FBetti7}
\xymatrix@C=15pt@R=5pt{
\cdots \ar [r]
&H^{1}(S_{2}, \Z) \ar [r]
&H^{1}({F}, \Z) \oplus H^{1}(\wt S_{2}, \Z) \ar[r] 
&H^{1}(F_{\wt S_{2}} \cup E_{\mu}, \Z)
\\
\ar [r]
&H^{2}(S_{2}, \Z) \ar [r]
&H^{2}({F}, \Z) \oplus H^{2}(\wt S_{2}, \Z) \ar[r] 
&H^{2}(F_{\wt S_{2}} \cup E_{\mu}, \Z) 
\\
\ar[r]
&H^{3}(S_{2}, \Z).
&
&
}
\end{align}
Combining $(\ref{eq:FBetti4})$, $(\ref{eq:FBetti6})$ and $(\ref{eq:FBetti7})$, we conclude that $B_1(F)=0$ and $B_2(F)=B_2(S_{2})+B_2(F_{\wt S_{2}} \cup E_{\mu}) - B_2(\wt S_{2})=B_2(S_{2}) + \sharp (C \cap F) -1$, which are the first and second assertions.
Since $3 \geq B_2(F) = B_2(S_{2}) + \sharp (C \cap F) -1 \geq B_{2}(S_{2})$, the third assertion follows.
\end{proof}

\begin{cor}\label{cor:projeq}
$S_{2}$ is projectively equivalent to $G_{i}$ for some $i \in \{1, \ldots, 12$, $14$, $15\}$ or $G_{13, p}$ for some $p \in T_{13}$.
\end{cor}

\begin{proof}
Combining Corollary \ref{cor:B-L} and Lemma \ref{lem:5.3}, we have the assertion.
\end{proof}

\begin{cor}\label{cor:intersection}
One of the following holds:
\begin{enumerate}
\item[\textup{(CU)}] $F$ is a cuspidal cubic, and $(B_{2}(S_{2}), \sharp(C \cap F))=(1,1)$.
\item[\textup{(L1)}] $F$ is a non-reduced line of length three, and $(B_{2}(S_{2}), \sharp(C \cap F))=(1,1)$.
\item[\textup{(QL)}] $F$ is the sum of a smooth conic and its tangent line at a point, and $(B_{2}(S_{2}), \sharp(C \cap F))=(2,1)$ or $(1,2)$.
\item[\textup{(L2)}] $F$ is the sum of a line and a non-reduced line of length two, and $(B_{2}(S_{2}), \sharp(C \cap F))=(2,1)$ or $(1, 2)$.
\item[\textup{(L3)}] $F$ is the sum of three lines intersecting in one point, and $(B_{2}(S_{2}), \sharp(C \cap F))=(3,1), (2,2)$ or $(1, 3)$.
\end{enumerate}
\end{cor}

\begin{proof}
Since $F$ is a plane cubic, the assertion follows from Lemma \ref{lem:5.3}.
\end{proof}

\begin{defi}
Let $\mathcal{L}(F_{\wt S_{2}}+E_{\mu})$ be the free $\Z$-module generated by irreducible components of $F_{\wt S_{2}}+E_{\mu}$.
The group homomorphism $\Theta: \mathcal{L}(F_{\wt S_{2}}+E_{\mu}) \rightarrow \mathrm{Pic}(\wt S_{2})$ is given as the quotient morphism by the linear equivalence.
\end{defi}

\begin{lem}\label{lem:theta}
The morphism $\Theta$ is surjective.
\end{lem}

\begin{proof}
We have the following exact sequence of homologies:
\begin{align}\label{eq:FBetti8}
\xymatrix@C=15pt@R=15pt{
H_{2}(F_{\wt S_{2}} \cup E_{\mu}, \Z) \ar [r]
&H_{2}(\wt S_{2} , \Z) \ar [r]
&H_{2}(\wt S_{2} , F_{\wt S_{2}} \cup E_{\mu}, \Z). 
}
\end{align}
We note that $H_{2}(F_{\wt S_{2}} \cup E_{\mu}, \Z) \cong \mathcal{L}(F_{\wt S_{2}} + E_{\mu})$.
Combining the universal coefficient theorem and the rationality of $\wt S_{2}$, we obtain
\begin{align}\label{eq:FBetti9}
H_{2}(\wt S_{2} , \Z) \cong H^2(\wt S_{2} , \Z) \cong \mathrm{Pic}(\wt S_{2}).
\end{align}
Combining the universal coefficient theorem and (\ref{eq:FBetti2}), we have 
\begin{align}\label{eq:FBetti10}
H_{2}(\wt S_{2} , F_{\wt S_{2}} \cup E_{\mu}, \Z)
\cong& \Z^{B_{2}(S_{2} \setminus F)} \oplus \{\text{torsion part of }H_{1}(S_{2} \setminus F, \Z)\}\\
\cong& 0.\nonumber
\end{align}
Therefore we can rewrite (\ref{eq:FBetti8}) as
\begin{align}\label{eq:FBetti11}
\xymatrix@C=15pt@R=15pt{
\mathcal{L}(F_{\wt S_{2}} + E_{\mu}) \ar [r]
&\Pic (\wt S_{2}) \ar [r]
&0,
}
\end{align}
which shows the surjectivity of $\Theta$.
\end{proof}

Now we can prove that $(S_{1}, S_{2})$ is projectively equivalent to one of the pairs listed in the case \textup{(d)} or \textup{(e)} of Theorem \ref{thm:main3tuple}.

\begin{prop}\label{prop:normB21}
Suppose that $B_{2}(S_{2})=1$.
Then $(S_{1}, S_{2})$ is projectively equivalent to 
$(\{y=0\}, G_{1})$ or $(\{y=0\}, G_{2})$.
\end{prop}

\begin{proof}
We may assume that $S_{2}=G_{i}$ for some $i \in \{1,2,3\}$ by Corollary \ref{cor:projeq}
and need only consider the five cases as in Corollary \ref{cor:intersection}.
Throughout the proof, we follow Notation \ref{nota:normrat}.

Suppose that the case (CU) holds.
Then Lemma \ref{lem:5.2} shows that the point $\Sing F$ contains $\Sing S_{2}$, which implies  $S_{2}=G_{1}$. 
Lemma \ref{lem:G1cusp} now yields $F_{\wt S_{2}} \sim H$. 
Hence $\Coker \Theta$ equals
\begin{align}
& \frac{\Z[H] \oplus \bigoplus_{i=1}^{6}\Z[E_{i}]}
{
\left(
\begin{array}{l}
H, H-E_{1}-E_{2}-E_{3}, E_{1}-E_{2}, \\
E_{2}-E_{3}, E_{3}-E_{4}, E_{4}-E_{5}, E_{5}-E_{6}
\end{array}
\right)
}
\cong \frac{\Z[E_{1}]}{(-3E_{1})}\cong  \Z/3\Z, \nonumber 
\end{align}
a contradiction with Lemma \ref{lem:theta}.

Suppose that the case (L1) holds.
Then Lemma \ref{lem:5.2} shows that 
$(S_{2}, F)=(G_{1},  3\la y, z \ra )$ or $(G_{2},  3\la y, z \ra )$. 
However it is easy to see that there is no element of $|-K_{G_{2}}|$ whose support is $\la y, z \ra $.
Hence the former holds and $(S_{1}, S_{2})=(\{y=0\}, G_{1})$. 

Suppose that the case (QL) holds.
Then Lemma \ref{lem:5.2} shows that the point $\Sing F$ contains $\Sing S_{2}$, which implies  $S_{2}=G_{1}$.
Lemma \ref{lem:G1curve} now yields that $F_{\wt S_{2}}$ is the sum of $E_{6}$ and a member of $|H-E_{1}|$.
Moreover, an easy computation shows that $S_{2} \setminus F \cong \wt S_{2} \setminus (F_{\wt S_{2}} \cup E_{\mu}) \cong \A^{1} \times \C^{*}$ and the $\P^{1}$-fibration on $\wt S_{2}$ associated with $|H-E_{1}|$ induces the second projection of $\A^{1} \times \C^{*}$.
On the other hand, \cite[Theorem 3.1]{K-Z} shows that the inclusion $C \setminus (C \cap S_{1}) \hookrightarrow S_{2} \setminus (S_{2} \cap S_{1})$ induces an isomorphism $H_{1}(C \setminus (C \cap S_{1}), \Z) \cong H_{1}(S_{2} \setminus (S_{2} \cap S_{1}), \Z)$.
Hence $(C_{\wt S_{2}} \cdot H-E_{1})=1$.
However, Lemma \ref{lem:G1curve} yields  $(C_{\wt S_{2}} \cdot H-E_{1})=0$, a contradiction.

Suppose that the case (L2) holds.
Then $S_{2} \neq G_{1}$ since $G_{1}$ has exactly one line.
Moreover, $S_{2} \neq G_{3}$ since all the three lines in $G_{3}$ are coplanar.
Hence $S_{2} =G_{2}$.
Since $G_{2}$ has exactly two lines $\la y, z \ra$ and $\la x, y \ra$, we conclude that $(S_{1}, S_{2}) =  (\{y=0\}, G_{2})$.

We note that the case (L3) cannot occur since both $G_{1}, G_{2}$ and $G_{3}$ do not contain three lines intersecting in one point.

Combining these results, we complete the proof.
\end{proof}

\begin{prop}\label{prop:normB22}
Suppose that $B_{2}(S_{2})=2$.
Then $(S_{1}, S_{2})$ is projectively equivalent to $(\{y=0\}, G_{4})$, $(\{z= \gamma y\}, G_{5})$ for some $\gamma \in \P^{1}$ or $(\{t=0\}, G_{6})$.
\end{prop}

\begin{proof}
We may assume that $S_{2}=G_{i}$ for some $i \in \{4,5,6,7,8\}$ by Corollary \ref{cor:projeq}
and need only consider the cases (QL), (L2) and (L3) as in Corollary \ref{cor:intersection}.
Throughout the proof, we follow Notation \ref{nota:normrat}.

Suppose that the case (QL) holds.
Take $L$ and $Q$ as the line and the conic respectively such that $F=L + Q$.
Then Lemma \ref{lem:5.2} shows that $L \cap Q \supset \Sing S_{2}$ and hence $S_{2}=G_{4}$ or $G_{5}$. 

Suppose that $S_{2} = G_{4}$ in addition. 
Since $\mu^{*}(F)-L_{\wt S_{2}}-Q_{\wt S_{2}}$ must be a sum of $(-2)$-curves, Table \ref{tb:negativecurve} shows that $(L_{\wt S_{2}}, Q_{\wt S_{2}}) \sim (E_{4}, H-E_{1})$ or $(E_{j}, H-E_{j})$ for some $j \in \{5, 6\}$.
If the former holds, then $\Coker \Theta$ equals
\begin{align}
& \frac{\Z[H] \oplus \bigoplus_{i=1}^{6}\Z[E_{i}]}
{
\left(
\begin{array}{l}
E_{4}, H-E_{1}, H-E_{1}-E_{2}-E_{3}, E_{1}-E_{2}, \\
E_{2}-E_{3}, E_{3}-E_{4}, H-E_{1}-E_{5}-E_{6}
\end{array}
\right)
}
\cong  \frac{\Z[E_{5}] \oplus \Z[E_{6}]}{(-E_{5}-E_{6})}\cong  \Z. \nonumber 
\end{align}
If the latter holds, then 
$\Coker \Theta$ equals
\begin{align}
& \frac{\Z[H] \oplus \bigoplus_{i=1}^{6}\Z[E_{i}]}
{
\left(
\begin{array}{l}
E_{j}, H-E_{j}, H-E_{1}-E_{2}-E_{3}, E_{1}-E_{2}, \\
E_{2}-E_{3}, E_{3}-E_{4}, H-E_{1}-E_{5}-E_{6}
\end{array}
\right)
}
\cong  \frac{\Z[E_{1}]}{(-3E_{1})}\cong  \Z/3\Z. \nonumber 
\end{align}
These contradict Lemma \ref{lem:theta}.

Hence $S_{2} = G_{5}$. 
Since $\Sing S_{2} \in L$, Table \ref{tb:line} shows that $S_{1}=\{z=\gamma y\}$ for some $\gamma \in \C$ or $\{x=\gamma y\}$ for some $\gamma \in \C^{*}$. 
If the latter case holds, then Lemma \ref{lem:G5curve} yields $L_{\wt S_{2}} \sim H-E_{1}-E_{6}$ and $Q_{\wt S_{2}} \sim H-E_{1}$ since $Q_{\wt S_{2}}$ is disjoint from $\la x, t \ra_{\wt S_{2}}=E_{6}$.
Hence $\Coker \Theta$ equals
\begin{align}
& \frac{\Z[H] \oplus \bigoplus_{i=1}^{6}\Z[E_{i}]}
{
\left(
\begin{array}{l}
H-E_{1}-E_{6}, H-E_{1}, H-E_{1}-E_{2}-E_{3},  \\
E_{1}-E_{2}, E_{2}-E_{3}, E_{3}-E_{4}, E_{4}-E_{5}
\end{array}
\right)
}
\cong  \frac{\Z[H]}{(-2H)}\cong  \Z/2\Z, \nonumber 
\end{align}
a contradiction with Lemma \ref{lem:theta}.
Therefore $(S_{1}, S_{2})=(\{z=\gamma y\}, G_{5})$ for some $\gamma \in \C$.

Suppose that the case (L2) holds.
Take $L_{1}$ and $L_{2}$ as lines such that $F=L_{1} + 2L_{2}$.
Since all the three lines in $G_{4}$ are coplanar, we have $S_{2} \neq G_{4}$.
We also have $L_{2} \supset \Sing S_{2}$ by Lemma \ref{lem:5.2}. 
Hence $S_{2}=G_{5}$ or $G_{6}$.

Suppose that $S_{2}=G_{5}$ in addition.
Then the support of $F$ is either $\la x,y \ra  \cup  \la x,t \ra$ or $\la y,z \ra  \cup  \la x,y \ra$.
We note that $\la x,y \ra_{\wt S_{2}} \sim H-E_{1}-E_{6}$ and $\la x,t \ra_{\wt S_{2}} \sim E_{6}$.
If the former holds, then $\Coker \Theta$ equals
\begin{align}
& \frac{\Z[H] \oplus \bigoplus_{i=1}^{6}\Z[E_{i}]}
{
\left(
\begin{array}{l}
H-E_{1}-E_{6}, E_{6}, H-E_{1}-E_{2}-E_{3},  \\
E_{1}-E_{2}, E_{2}-E_{3}, E_{3}-E_{4}, E_{4}-E_{5}
\end{array}
\right)
}
\cong  \frac{\Z[H]}{(-2H)}\cong  \Z/2\Z, \nonumber 
\end{align}
a contradiction with Lemma \ref{lem:theta}. 
Hence $(S_{1}, S_{2})=(\{y=0\}, G_{5})$.

Suppose that $S_{2}=G_{6}$ in addition.
Then $L_{2}=\la y, t\ra$ since $L_{2} \supset \Sing S_{2}$
and $L_{1}=\la x, t \ra$ because $\la y, z\ra, \la x, y\ra$ and $\la y, t\ra$ are coplanar. 
Hence $(S_{1}, S_{2})=(\{t=0\}, G_{6})$.

Suppose that the case (L3) holds.
Then $S_{2}=G_{4}$ or $G_{5}$ because the point $\Sing F$ contains $\Sing S_{2}$ by Lemma \ref{lem:5.2}. 
Since all the three lines in $G_{5}$ are not coplanar, we have $S_{2} \neq G_{5}$.
Hence $(S_{1}, S_{2})=(\{y=0\}, G_{4})$.

Combining these results, we complete the proof.
\end{proof}

\begin{prop}\label{prop:normB23}
Suppose that $B_{2}(S_{2})=3$.
Then $(S_{1}, S_{2})$ is projectively equivalent to $(\{y=0\}, G_{9})$, $(\{y=0\}, G_{10})$ or $(\{x=t\}, G_{11})$.
\end{prop}

\begin{proof}
We may assume that $S_{2}=G_{i}$ for some $i \in \{9, \ldots 12, 14, 15\}$ or $G_{13, p}$ for some $p \in T_{13}$ by Corollary \ref{cor:projeq}.
Moreover, the case (L3) holds by Corollary \ref{cor:intersection}.
Lemma \ref{lem:5.2} now shows that the point $\Sing F$ is the same as $\Sing S_{2}$.
Therefore Tables \ref{tb:sing} and \ref{tb:line} give the assertion.
\end{proof}

\section{The case $p_{a}(C)=0$ and $S_{2}$ is non-normal}\label{sec:nonnormal}

In this section, 
we assume the following:
\begin{assu}
$p_{a}(C)=0$, $U$ is an affine homology $3$-cell and $S_{2}$ is non-normal.
\end{assu}

Then $S_{2}$ is projectively equivalent to $R_{i}$ for some $i \in \{1, 2, 3, 4\}$ by Theorem \ref{thm:nnorm1}.
In this section, we show that $(C, S_{1}, S_{2})$ belongs to one of the six cases of Theorem \ref{thm:main3tuple} when $i \leq 2$, and that $(S_{1}, S_{2})$ is projectively equivalent to one of the pairs listed in the case \textup{(d)} or \textup{(e)} of Theorem \ref{thm:main3tuple} when $i \geq 3$.


\begin{prop}\label{prop:R1}
Suppose that $S_{2}=R_{1}$.
Then one of the following occurs:
\begin{itemize}
\item $S_{1}=\{x=0\}$.
\item $S_{1}=\{z=0\}$ and $\deg C=1$.
\item $S_{1}=\{y=\gamma x\}$ for some $\gamma \in \A^1$, $\deg C=3$ or $4$, and $\sharp (F \cap C)=B_{2}(F)$.
\end{itemize}
In particular, the second case gives the triplet $(C, S_{1}, S_{2})$ projectively equivalent to the subvarieties as in Example \ref{ex:nonA3}.
\end{prop}

\begin{proof}
Since $p_{a}(C)=0$, Lemma \ref{lem:3.6} shows that $\deg C=1,3$ or $4$.
Moreover, $F$ is either a cuspidal cubic or the sum of rulings. 

Suppose that $F$ is a cuspidal cubic. 
Then Lemma \ref{lem:2.10} (1) yields $N_{1}+N_{2}-N_{1 \cap 2}=1$.
On the other hand, we have $N_{2}=\sharp(C \cap E) =\sharp(C \cap \{x=0\}) \geq 1$.
Hence $N_{1}=N_{2}=N_{1 \cap 2}=1$ by Lemma \ref{lem:2.10} (2).
Since $F \cap E$ also consists of one point, say $p$, we conclude that $C \cap F=C \cap E=F \cap E=\{p\}$.
We note that $p$ is distinct from the vertex of $S_{2}$.

If $\deg C=1$, then $C=E$ since $C$ is the ruling containing $p$, a contradiction with the assumption in Problem \ref{prob:main}. 
Hence $\deg C  \geq 3$. 
Since 
$(C_{\wt S_{2}} , F_{\wt S_{2}}) =\deg C \geq 3$, 
the curve $C_{\wt S_{2}}$ intersects with $F_{\wt S_{2}}$ at $\mu^{-1}(p)$ tangentially.
Since $\mu$ contracts the tangent direction of $F_{\wt S_{2}}$ at $\mu^{-1}(p)$, 
it also contracts that of $C_{\wt S_{2}}$ at $\mu^{-1}(p)$.
Hence $C$ is singular at $p$, a contradiction.

Therefore $F$ is the sum of rulings.
Lemma \ref{lem:2.10} (1) now shows that
\begin{align}\label{eq:cuspcone}
N_{1}+N_{2} -N_{1\cap 2}=B_{2}(F).
\end{align}

Suppose that $\deg C =1$.
Then both $C \cap F$ and $C \cap E$ is the vertex of $S_{2}$.
Hence $N_{1}=N_{2}=N_{1 \cap 2}=1$ and $B_{2}(F)=1$.
Therefore $S_{1}=\{x=0\}$ or $\{z=0\}$.

Suppose that $\deg C  = 3$.
Then $N_{1}=B_{2}(F)$ and $N_{2}=1$.
On the other hand, $N_{1 \cap 2}=1$ if $E \subset F$ and $0$ otherwise.
(\ref{eq:cuspcone}) now shows that $E \subset F$ and hence $S_{1}=\{y= \gamma x\}$ for some $ \gamma \in \P^{1}$. 

Suppose that $\deg C =4$. Since $C_{\wt S_{2}}  \sim  \Sigma_{3}  +4f_{3}$, we have
\begin{align}
N_{1}&=
\left\{
\begin{array}{ll}
B_{2}(F)+1&\mathrm{if}\: F_{\wt S_{2}} \cap C_{\wt S_{2}}  \cap \Sigma_{3}= \emptyset \\
B_{2}(F)  &\mathrm{if}\: F_{\wt S_{2}} \cap C_{\wt S_{2}} \cap \Sigma_{3} \neq \emptyset,\\
\end{array}
\right.
\\
N_{2}&=
\left\{
\begin{array}{ll}
2\ \ \ \ \ \ \ \ \ \ \ \ \ &\mathrm{if}\: \wt E \cap C_{\wt S_{2}}   \cap \Sigma_{3}= \emptyset \\
1&\mathrm{if}\: \wt E \cap C_{\wt S_{2}}   \cap \Sigma_{3} \neq \emptyset,\\
\end{array}
\right.
\\
N_{1 \cap 2}&=
\left\{
\begin{array}{ll}
N_{2}\ \ \ \ \ \ \ \ \ \ \ &\mathrm{if}\: F \supset E\\
1&\mathrm{if}\:F \not \supset E.\\
\end{array}
\right. 
\end{align}

Assume that $F \not \supset E$. 
Then (\ref{eq:cuspcone}) implies $(N_{1}, N_{2})=(B_{2}(F), 1)$. 
Hence $F_{\wt S_{2}} \cap (C_{\wt S_{2}} \cap \Sigma_{3}) \neq \emptyset$ and $\wt E \cap (C_{\wt S_{2}} \cap \Sigma_{3}) \neq \emptyset$.
This implies, however, that $\wt E \subset \F_{3}$ is the fiber of $\P^{1}$-bundle passing through the point $C_{\wt S_{2}} \cap \Sigma_{3}$ and $F_{\wt S_{2}}$ contains such a fiber.
Hence $F \supset E$, a contradiction. 

Therefore $F \supset E$ and $S_{1}=\{y=\gamma x\}$ for some $\gamma \in \P^1$.
(\ref{eq:cuspcone}) now shows that $\sharp(C \cap F)=B_{2}(F)$.

Combining these results, we obtain the first assertion.
Now suppose that $S_{1}=\{z=0\}$ and $\deg C=1$.
Then $C=\{y=ax, z=-a^{2}y\}$ for some $a \in \C^{*}$.
Now take $f$ as the automorphism of $\P^{3}_{[x:y:z:t]}$ such that $f^{*}(x)=x$, $f^{*}(y)=a^{-1}y$, $f^{*}(z)=a^{-3} z$ and $f^{*}(t)=t$.
Then $f(S_{i})=S_{i}$ for $i \in \{1, 2\}$ and $f(C)=\{x=y=-z\}$, and the second assertion holds.
\end{proof}

\begin{prop}\label{prop:R2}
Suppose that $S_{2}=R_{2}$.
Then $S_{1}=\{y= \gamma x\}$ for some $\gamma \in \P^1$, $\deg C =3$ or $4$, and $ \sharp (F \cap C)=B_{2}(F)+1$. 
\end{prop}

\begin{proof}
An analysis similar to that in the proof of Proposition \ref{prop:R1} shows that $\deg C=1,3$ or $4$ and $F$ is the sum of rulings.
Lemma \ref{lem:2.10} (1) now shows that
\begin{align}\label{eq:nodalcone}
N_{1}+N_{2} -N_{1\cap 2}=B_{2}(F)+1 \geq 2.
\end{align}

Suppose that $\deg C =1$.
Then both $C \cap F$ and $C \cap E$ is the vertex of $S_{2}$.
Hence $N_{1}=N_{2}=N_{1 \cap 2}=1$, a contradiction with (\ref{eq:nodalcone}).

Suppose that $\deg C  = 3$ and $F \not \supset E$.
Then $N_{1}=B_{2}(F)$,  $N_{2}=2$ and $N_{1 \cap 2}=0$, a contradiction with (\ref{eq:nodalcone}).

Suppose that $\deg C =4$ and $F \not \supset E$. 
Then we have
\begin{align}
N_{1}&=
\left\{
\begin{array}{ll}
B_{2}(F)+1&\ \mathrm{if}\: F_{\wt S_{2}} \cap C_{\wt S_{2}}  \cap \Sigma_{3}= \emptyset \\
B_{2}(F)   &\ \mathrm{if}\: F_{\wt S_{2}} \cap C_{\wt S_{2}}  \cap \Sigma_{3} \neq \emptyset,\\
\end{array}
\right.
\\
N_{2}&=
\left\{
\begin{array}{ll}
3\ \ \ \ \ \ \ \ \ \ \ \ \ \ &\mathrm{if}\: \wt E \cap C_{\wt S_{2}}  \cap \Sigma_{3}= \emptyset \\
2&\mathrm{if}\: \wt E \cap C_{\wt S_{2}}  \cap \Sigma_{3} \neq \emptyset,\\
\end{array}
\right.
\\
N_{1 \cap 2}&=1.
\end{align}
As in the proof of Proposition \ref{prop:R1}, these equations yield $F \supset E$, a contradiction.

Hence $\deg C \geq 3$ and $F \supset E$.
In particular, $S_{1}=\{y=\gamma x\}$ for some $\gamma \in \P^1$ and  $\sharp(C \cap F)=B_{2}(F)+1$ by (\ref{eq:nodalcone}).
\end{proof}

Next we prove lemmas needed to treat the case where $S_{2}=R_{3}$.

\begin{lem}\label{lem:5.18} 
Suppose that $S_{2}=R_{3}$. Then the following hold:
\begin{enumerate}
\item[\textup{(1)}]  The morphism $\mu_*:H_{2}(\F_1, \Z) \rightarrow H_{2}(S_{2}, \Z)$ maps $[\Sigma_{1}] +[f_{1}]$ to $2[E]$.
\item[\textup{(2)}] $H_{i}(S_{2}, \Z)=\Z$ for  $i=0,4$, $\Z^2$ for $i=2$, and $0$ otherwise.
\item[\textup{(3)}] $\{ [E], \mu_{*}[f_{1}] \}$ is a $\Z$-basis of $H_{2}(S_{2}, \Z)$.
\end{enumerate}
\end{lem}

\begin{proof}
\noindent(1) For $i \geq 0$, we have the following commutative diagram with exact rows:
\begin{align}\label{diag:R3}
\xymatrix@C=15pt@R=15pt{
\cdots \ar[r]
&{H_{i+1}(\ol{E},\F_1,\Z)} \ar [r] \ar [d]^-{\zeta_{i+1}}
&{H_{i}(\ol{E}, \Z)} \ar [r]^-{\delta_{i}} \ar [d]^-{\theta_{i}} 
&{H_{i}(\F_1, \Z)} \ar [r]^-{\epsilon_{i}} \ar [d]^-{\mu_{*}} 
&{H_{i}(\ol{E}, \F_1, \Z)} \ar [r] \ar [d]^-{\zeta_{i}} 
&\cdots\\
\cdots \ar[r]
&{H_{i+1}(E,S_{2}, \Z)}\ar [r]
&{H_{i}(E, \Z)} \ar [r]^-{\delta'_{i}}
&{H_{i}(S_{2}, \Z)} \ar [r]^-{\epsilon'_{i}}
&{H_{i}(E, S_{2}, \Z)} \ar [r]
&\cdots.
}
\end{align}
Then $\theta_{2}= 2 \times \mathrm{id}_{\Z}$ by the construction of $R_{3}$.
Hence $\mu_{*}([\Sigma_{1}]+[f_{1}])=\mu_{*}[\ol E]=2[E]$.

\noindent (2)
By \cite[Chap.4, Sect.8, Theorem 9]{sp66}, $\zeta_{i}$ is an isomorphism for $i \geq 0$. 
By Lefschetz duality, we have $H_{i}(\ol{E}, \F_1, \Z) \cong H^{4-i}_{c}(\F_{1} \setminus \ol{E}, \Z) \cong H^{4-i}(\P^{1}, \Z)$.
Hence the bottom of (\ref{diag:R3}) gives the assertion.

\noindent (3)
Since $H_{2}(\F_1, \Z)$ is generated by fundamental classes $[\ol{E}]=[ \Sigma_{1} ]+[f_{1}]$ and $[f_{1}]$, $H_{2}(\ol{E},\F_1,\Z)$ is identified as the free $\Z$-module generated by $[f_{1}]$.
Hence $\epsilon_{2}$ has a section, say $s$. 
Therefore $\mu_{*} \circ s \circ \zeta_{2}^{-1}$ is a section of $\epsilon'_{2}$ and we have the assertion. 
\end{proof}

\begin{lem}\label{lem:5.19}
Suppose that $S_{2}=R_{3}$ and $N_{1}+N_{2}-N_{1 \cap 2}=1$.
Then the canonical map $\xi \colon H_{2}(F, \Z) \rightarrow H_{2}(S_{1},  \Z) \oplus H_{2}(S_{2}, \Z)$ is injective.
If $H_{1}(F, \Z)=0$ in addition, then $\Coker \xi \cong \Z$.
\end{lem}

\begin{proof}
By Lemma \ref{lem:2.10} (2), we have $(N_{1}, N_{2}, N_{1 \cap 2})=(1, 0, 0)$ or $(1,1,1)$.
In particular $\sharp (C \cap (E \cup F))=1$.
Hence $G \coloneqq E_{\vp} \cap U$ is isomorphic to $\A^{2}$.
Applying the Thom isomorphism to the pair $(U, G)$, we get the following exact sequence:

\begin{align}
\xymatrix@C=15pt@R=15pt{
{\cdots} \ar [r]
&{H_{i-1}(\A^{2}, \Z)} \ar [r] 
&{H_{i}(U \setminus G, \Z)} \ar [r] 
&{H_{i}(U, \Z)} 
\ar [r] 
& {H_{i-2}(\A^{2}, \Z)} \ar [r] 
&{\cdots}.
}
\end{align}
Since $U \setminus G \cong \P^3 \setminus (S_{1} \cup S_{2})$, this sequence shows that $H_{i}(\P^3 \setminus (S_{1} \cup S_{2}),\Z) \cong 0$ for $i \geq 2$. 
On the other hand, the Lefschetz duality gives the following exact sequence:

\begin{align}
\xymatrix@C=15pt@R=15pt{
{\cdots} \ar [r]
&{H_{6-i}(\P^3 \setminus (S_{1} \cup S_{2}), \Z)} \ar [r]
&{H^{i}(\P^3, \Z)} \ar [r] 
&{H^{i}(S_{1} \cup S_{2}, \Z)} 
\\
\ar [r] 
&{H_{6-(i+1)}(\P^3 \setminus (S_{1} \cup S_{2}), \Z)} \ar [r] 
&{\cdots}.
&
}
\end{align}

Therefore $H^2(S_{1} \cup S_{2}, \Z)= \Z$ and $ H^3(S_{1} \cup S_{2}, \Z)=0$.
By the universal coefficient theorem, we obtain $H_{2}(S_{1} \cup S_{2}, \Z)=\Z$ and $H_{3}(S_{1} \cup S_{2}, \Z)$ is a torsion. 
Finally let us consider the following Mayer-Vietoris exact sequence:

\begin{align}\label{diag:R3-2}
\xymatrix@C=15pt@R=15pt{
&\cdots \ar[r]
&H_{3}(S_{1}, \Z) \oplus H_{3}(S_{2}, \Z) \ar [r] 
&{H_{3}(S_{1} \cup S_{2}, \Z)} \\
\ar [r] 
&{H_{2}(F, \Z)} \ar^-{\xi}[r]
&H_{2}(S_{1}, \Z) \oplus H_{2}(S_{2}, \Z) \ar [r] 
&{H_{2}(S_{1} \cup S_{2}, \Z)}  \\
\ar [r]
&{H_{1}(F, \Z)} \ar [r]
&\cdots.\hspace{80pt}
&\\
} 
\end{align}
By Lemma \ref{lem:5.18} (2), we have $H_{3}(S_{1}, \Z) \cong H_{3}(S_{2}, \Z) \cong 0$.
Then $H_{3}(S_{1} \cup S_{2}, \Z)$ is a submodule of a free $\Z$-module $H_{2}(F, \Z)$.
Hence $H_{3}(S_{1} \cup S_{2}, \Z) \cong 0$.
Now (\ref{diag:R3-2}) gives the following exact sequence:
\begin{align}
\xymatrix@C=15pt@R=15pt{
0 \ar [r] 
&{H_{2}(F, \Z)} \ar^-{\xi}[r]
&H_{2}(S_{1}, \Z) \oplus H_{2}(S_{2}, \Z) \ar [r] 
&\Z \ar [r]
&{H_{1}(F, \Z)},
} 
\end{align}
and the assertions hold.
\end{proof}

Now we can determine $S_{1}$ in the case where $S_{2}=R_{3}$.
 
 \begin{prop}\label{prop:R3}
Suppose that $S_{2}=R_{3}$.
Then $S_{1}=\{y=\gamma x\}$ for some $\gamma \in \P^1$.
\end{prop}

\begin{proof}
It suffices to show that $F \supset E$.
Conversely, suppose that $F \not \supset E$.
By Lemma \ref{lem:2.10}, we have $\eu(F)=2+N_{1}+N_{2}-N_{1 \cap 2} \geq 3$. 
Hence $F$ is one of the following:
\begin{enumerate}
\item The sum of a line and a non-reduced line of length two, 
\item The sum of a smooth conic and its tangent line at a point, 
\item The sum of three lines.
\end{enumerate}
The former two cases imply that $N_{1}+N_{2}-N_{1 \cap 2}=\eu(F)-2=1$ and $H_{1}(F, \Z)=0$.
On the other hand, Lemma \ref{lem:3.7} shows that each line in $S_{2}$ is either $E$, $\mu_{*}(\Sigma_{1})$ or the image of a member of $|f_{1}|$ in $\wt S_{2} \cong \F_{1}$ by $\mu$.
Since $F \not \supset E$, the case (3) also implies that $N_{1}+N_{2}-N_{1 \cap 2}=\eu(F)-2=1$.

Suppose that the case (1) holds. 
Let $F=L_{1}+2L_{2}$ be the irreducible decomposition. 
Then $(L_{1})_{\wt S_{2}} \sim \Sigma_{1}$ and $(L_{2})_{\wt S_{2}} \sim f_{1}$ because Theorem \ref{thm:nnorm2} shows that $\mu^*(-K_{S_{2}}) \sim  \Sigma_{1} +2f_{1}$.
Hence
\begin{align}
\Coker \xi = 
\frac{H_{2}(S_{1}, \Z) \oplus H_{2}(S_{2}, \Z)}{(([\mc{O}_{S_{1}}(1)], 2[E]-\mu_{*}[f_{1}]), ([\mc{O}_{S_{1}}(1)], \mu_{*}[f_{1}]))}
\cong 
\Z \oplus \Z/2\Z, \nonumber
\end{align}
a contradiction with Lemma \ref{lem:5.19}.

Suppose that the case (2) holds.
Let $L, Q$ be the line and the conic respectively such that $F=L + Q$.
Then $L_{\wt S_{2}} \sim f_{1}$ and $Q_{\wt S_{2}} \sim \Sigma_{1} + f_{1}$ since $\mu^*(-K_{S_{2}}) \sim \Sigma_{1} +2f_{1}$.
Hence
\begin{align}
\Coker \xi = \frac{H_{2}(S_{1}, \Z) \oplus H_{2}(S_{2}, \Z)}{(([\mc{O}_{S_{1}}(1)], \mu_{*}[f_{1}]), ([\mc{O}_{S_{1}}(2)], 2[E]))}
\cong 
\Z \oplus \Z/2\Z, \nonumber
\end{align}
a contradiction with Lemma \ref{lem:5.19}.

Suppose that the case (3) holds.
Then there are lines $L_{1}$, $L_{2}$, $L_{3}$ such that $F=L_{1}+L_{2}+L_{3}$, $(L_{1})_{\wt S_{2}} \sim (L_{2})_{\wt S_{2}} \sim f_{1}$ and $(L_{3})_{\wt S_{2}} \sim \Sigma_{1}$.
However $\xi([L_{1}])=\xi([L_{2}])=\xi(\mu_{*}[f_{1}])$, a contradiction with Lemma \ref{lem:5.19}.

Therefore $F \supset E$, which implies $S_{1}=\{y=\gamma x\}$ for some $\gamma \in \P^1$.
\end{proof}

\begin{prop}\label{prop:R4}
Suppose that $S_{2}=R_{4}$.
Then $S_{1}=\{y=\gamma x\}$ for some $\gamma \in \P^{1}$.
\end{prop}

\begin{proof}
By Lemma \ref{lem:2.10}, we have $\eu(F)=1+N_{1}+N_{2}-N_{1 \cap 2} \geq 2$. 
Hence $F$ is one of the following:
\begin{enumerate}
\item A cuspidal cubic, 
\item A non-reduced line of length three, 
\item The sum of a smooth conic and a line, 
\item The sum of a line and non-reduced line of length two, 
\item The sum of three lines.
\end{enumerate} 
There is also the irreducible decomposition $\wt E =\wt E_{1} + \wt E_{2}$ such that $\wt E_{1} \sim \Sigma_{1}$ and $\wt E_{2}  \sim f_{1}$ by Lemma \ref{lem:2.16}. 

Suppose that the case (1) holds.
Then $N_{1}+N_{2}-N_{1 \cap 2}=\eu(F)-1=1$ and $N_{2} = \sharp (C \cap E)= \sharp (C \cap \{x=0\}) \geq 1$. 
Lemma \ref{lem:2.10} (2) now shows that  
\begin{align}\label{eq:R4cusp1}
N_{1} =N_{2}=N_{1 \cap 2}=1.
\end{align}
On the other hand, we have $F_{\wt S_{2}} \sim \mu^{*}(-K_{S_{2}}) \sim \Sigma_{1}+2f_{1}$.
Since $F$ has a cuspidal singularity, both $F_{\wt S_{2}} \cap \wt E_{1}$ and $F_{\wt S_{2}} \cap \wt E_{2}$ coincide with the unique point in $\wt{E}_1 \cap \wt{E}_2$, say $\wt p$.
In particular $F \cap E=\{p\} \coloneqq \{\mu(\wt p)\}$.
Combining this result and (\ref{eq:R4cusp1}), we obtain $C \cap E = C \cap F = \{p\}$.
Since $C$ is smooth, we also obtain $C_{\wt S_{2}}  \cap F_{\wt S_{2}}=C_{\wt S_{2}}  \cap \wt{E}=\{\wt p\}$. 

Now write $C_{\wt S_{2}} \sim a \Sigma_{1}+ bf_{1}$ with $a, b \in \Z_{\geq 0}$.
Then $(a, b) \neq (1, 0)$ or $(0, 1)$ since $C_{\wt S_{2}} \not\subset \wt E$.
Hence $C_{\wt S_{2}}$ intersects with $F_{\wt S_{2}}$ at $\wt p$ with multiplicity $a+b \geq 2$.
Since $\mu$ contracts the tangent direction of $F_{\wt S_{2}}$ at $\wt p$, 
it also contracts that of $C_{\wt S_{2}}$ at $\wt p$.
Hence $C$ is singular at $p$, a contradiction.

Suppose that the case (3) holds. 
Let $L$ and $Q$ be the line and the conic respectively such that  $F=L+Q$. 
We note that $L \neq E$ since $S_{2}$ is smooth at the generic point of $L$.
Then $L_{\wt S_{2}} \sim f_{1}$ and $Q_{\wt S_{2}} \sim \Sigma_{1} + f_{1}$ by Lemma \ref{lem:3.7}.
Let $L_{\wt S_{2}} \cap Q_{\wt S_{2}}=\{\wt q\}, L_{\wt S_{2}} \cap \wt{E}_1=\{\wt p_{1}\}$ and $Q_{\wt S_{2}} \cap \wt{E}_2=\{\wt p_2\}$. 
Since $L$ intersects with $Q$ at $\mu( \wt q)$ transversally, we have $\mu(\wt p_1)=\mu(\wt p_2) \eqqcolon p$ and hence $\eu(F)=2$.

An analysis similar to that in the argument on the case (1) shows that $C \cap E = C \cap F = E \cap F=\{p\}$. 
Hence both $C_{\wt S_{2}}  \cap (L_{\wt S_{2}} \cup Q_{\wt S_{2}})$ and $C_{\wt S_{2}}  \cap (\wt E_{1} \cup \wt E_{2})$ are contained in $\{\wt p_1, \wt p_2\}$. 

Now write $C_{\wt S_{2}}=a \Sigma_{1} +b f_{1}$ with $a, b \in \Z_{\geq 0}$.
Then $(a, b) \neq (1, 0)$ or $(0, 1)$ since $C_{\wt S_{2}} \not\subset \wt E$.
If $a=b$, then $C_{\wt S_{2}} \cap L_{\wt S_{2}}$ contains a point distinct from $\wt p_{1}$ and $\wt p_{2}$, a contradiction.
Hence $0<a < b$. 
Since $C_{\wt S_{2}}$ intersects with both $\wt E_{1}$ and $\wt E_{2}$, it must contain $\wt p_{1}$ and $\wt p_{2}$.
Hence $p \in C$ is a nodal singularity, a contradiction.

Therefore $F$ consists of lines.
Since $\mu^{*}(F) \sim \Sigma_{1}+2f_{1}$, Lemma \ref{lem:3.7} now shows that $F \supset E$.
Hence $S_{1}=\{y=\gamma x\}$ for some $\gamma \in \P^{1}$.
\end{proof}

\section{Contractibilities}\label{sec:cont}

In \S\ref{sec:normal} and \ref{sec:nonnormal}, we show that if $p_{a}(C)=0$ and $U$ is an affine homology $3$-cell, then either $(S_{1}, S_{2})$ is projectively equivalent to one of the pairs listed in the cases \textup{(d)} or \textup{(e)}, or $(C, S_{1}, S_{2})$ belongs to one of the cases \textup{(b)}, \textup{(c)} and \textup{(f)}.
In this section, we seek equivalent conditions for $U$ to be an affine homology $3$-cell in the former situation, and we show that $U$ is a contractible affine $3$-fold in the latter situation.

\begin{prop}\label{prop:A2}
Suppose that $p_{a}(C)=0$ and $(S_{1}, S_{2})$ is projectively equivalent to one of the pairs listed in the case \textup{(d)}.
Then the following are equivalent:
\begin{enumerate}
\item[\textup{(1)}] $U$ is an affine homology $3$-cell.
\item[\textup{(2)}] $\sharp (C \cap S_{1})=1$.
\item[\textup{(3)}] $U \cong \A^{3}$.
\end{enumerate}
\end{prop}

\begin{proof} 
Since $(3) \Rightarrow (1)$ is trivial, it suffices to show that $(1) \Rightarrow (2) \Rightarrow (3)$.

Let $G \coloneqq  E_{\vp} \cap U$ and $S \coloneqq  S_{2} \setminus F$. 
Then $G \cong \A^{1} \times (C \setminus (C \cap S_{1}))$. 
For each case, it is also easy to check that there is a coordinate $\{x, y, z\}$ of $ \P^3 \setminus S_{1} \cong \A^{3}$ such that $S=\{x=0\}$. 
In particular, $S \cong \A^{2}$.

\noindent $(1) \Rightarrow (2)$: By \cite[Theorem 3.1]{K-Z}, we have an isomorphism of homology rings 
$\Z \cong H_{*}(S, \Z) \cong H_{*}(G, \Z) \cong H_{*}(C \setminus (C \cap S_{1}), \Z)$. 
Hence $\sharp (C \cap S_{1})=1$.

\noindent$(2) \Rightarrow (3)$: The assumption implies that $C \setminus (C \cap S_{1}) \cong \A^{1}$. 
Then we may assume that $C \setminus (C \cap S_{1}) = \{x=y=0\}$ by the Abhyankar-Moh-Suzuki theorem \cite{Suz74, A-M75}.
Hence $U$ is the affine modification of $\A^{3}$ with the locus $(\{x=y=0\} \subset \{x=0\})$, which is isomorphic to $\A^{3}$.
\end{proof}

\begin{defi}
In what follows, $\Delta \colon H_{1}(C \setminus (C \cap S_{1}), \Z) \rightarrow H_{1}(S_{2} \setminus F, \Z)$ stands for the homomorphism induced by the inclusion $C \setminus (C \cap S_{1}) \hookrightarrow S_{2} \setminus F$.
\end{defi}

\begin{prop}\label{prop:nonA2}
Suppose that $p_{a}(C)=0$ and $(S_{1}, S_{2})$ is projectively equivalent to one of the pairs listed in the case \textup{(e)}.
Then the following are equivalent:
\begin{enumerate}
\item[\textup{(1)}] $U$ is an affine homology $3$-cell.
\item[\textup{(2)}] $\Delta$ is an isomorphism.
\item[\textup{(3)}] $U$ is a contractible affine $3$-fold.
\end{enumerate}
In particular, if $U$ is an affine homology $3$-cell, then $\sharp (C \cap S_{1})=2$.
\end{prop}

\begin{proof}
Let $G \coloneqq  E_{\vp} \cap U$ and $S \coloneqq  S_{2} \setminus F$. 
Then $G \cong \A^{1} \times (C \setminus (C \cap S_{1}))$.
For each case, it is also easy to check that $S \cong \A^{1} \times \C^*$.
In particular $H_{i}(G, \Z) \cong H_{i}(S, \Z) \cong 0$ for all $i \geq 2$. 
Since $\vp_*:H_{1}(G, \Z) \rightarrow H_{1}(S, \Z)$ factors through $\Delta$, \cite[Theorem 3.1, Corollary 3.1]{K-Z} now yields
\begin{align*}
 &U \text{ is an affine homology $3$-cells}\\ 
  \iff & \vp_*:H_{*}(G, \Z) \rightarrow H_{*}(S, \Z) \text{ is an isomorphism}\\
  \iff & \vp_*:H_{1}(G,\Z) \rightarrow H_{1}(S, \Z) \text{ is an isomorphism}\\
  \iff & \Delta \text{ is an isomorphism}\\
  \iff & U \text{ is a contractible affine $3$-fold},
\end{align*}
which is the first assertion.
The second assertion follows from $H_{1}(S, \Z) \cong H_{1}(\A^{1} \times \C^{*}, \Z) \cong \Z$.
\end{proof}

\begin{prop}\label{prop:cone} 
Suppose that one of the cases \textup{(b)} and \textup{(c)} holds.
Then $U$ is a contractible affine $3$-fold.
\end{prop}

\begin{proof}
For each case, $\vp_*:H_{1}(E_{\vp} \cap U, \Z) \rightarrow H_{1}(S_{2} \setminus F, \Z)$ factors through $\Delta$ and is an isomorphism by the condition on $\sharp(C \cap S_{1})$.
Hence the assertion follows from \cite[Corollary 3.1]{K-Z}.
\end{proof}

Now we can prove Theorem \ref{thm:main3tuple}.

\begin{proof}[Proof of Theorem \ref{thm:main3tuple}]
First suppose that $p_{a}(C) \geq 1$.
If $U$ is an affine homology $3$-cell, then the case \textup{(a)} holds by Proposition \ref{prop:highgenus}.
On the other hand, if the case \textup{(a)} holds, then $U \cong \A^{3}$ by Propositions \ref{prop:elldeg3} and \ref{prop:elldeg4}.

In the remainder of the proof, we may assume that $p_{a}(C)=0$.
Suppose that $S_{2}$ is normal.
Then $B_{2}(S_{2}) \leq 3$ by Lemma \ref{lem:5.3}.
Combining Propositions \ref{prop:normB21}--\ref{prop:normB23}, \ref{prop:A2} and \ref{prop:nonA2}, we conclude that $U$ is an affine homology $3$-cell if and only if one of the cases \textup{(d)} and \textup{(e)} holds.

Suppose that $S_{2}$ is non-normal.
Then we may assume that $S_{2} = R_{i}$ for some $i \in \{1,2,3,4\}$ by Theorem \ref{thm:nnorm1}.

Firstly suppose that $i=1$.
If $U$ is an affine homology $3$-cell, then Proposition \ref{prop:R1} shows that one of the cases \textup{(b)}, \textup{(d)} and \textup{(f)} holds.
On the other and, 
if the case \textup{(b)} holds, then $U$ is a contractible affine $3$-fold by Proposition \ref{prop:cone}.
If the case \textup{(d)} holds, then $U \cong \A^{3}$ by Proposition \ref{prop:A2}.
If the case \textup{(f)} holds, then $U \cong \A^{1} \times W(3, 2)$ by Example \ref{ex:nonA3}.

Next suppose that $i=2$.
Then $U$ is an affine homology $3$-cell if and only if the case \textup{(c)} holds by Propositions \ref{prop:R2} and \ref{prop:cone}.

Finally suppose that $i \geq 3$.
Then $U$ is an affine homology $3$-cell if and only if one of the cases \textup{(d)} and \textup{(e)} holds by Propositions \ref{prop:R3}
--\ref{prop:A2} and \ref{prop:nonA2}.
 
Combining these results, we complete the proof.
\end{proof}

\section{Isomorphism classes}\label{sec:isomclass}

In this section, we prove Theorem \ref{thm:mainisom}.
When the case \textup{(a)} (resp.\ \textup{(d)}) of Theorem \ref{thm:main3tuple} holds, then Propositions \ref{prop:elldeg3} and \ref{prop:elldeg4} (resp.\ Proposition \ref{prop:A2}) yields $U \cong \A^{3}$.
For this reason, we will treat only the cases \textup{(b)}, \textup{(c)} and \textup{(e)}.

\begin{lem}\label{lem:G2deg2}
Suppose that the case \textup{(e)} of Theorem \ref{thm:main3tuple} holds with $(S_{1}, S_{2}) = (\{y=0\}, G_{2})$ and $\deg C \leq 2$. 
Then $U \cong \A^{3}$.
\end{lem}


\begin{proof}
First let us determine the linear equivalence class of $C$ using Notation \ref{nota:normrat}.
It is easy to check that $S_{2} \setminus F \cong \wt S_{2} \setminus (F_{\wt S_{2}} \cup E_{\mu}) \cong \A^{1} \times \C^{*}$ and 
the $\P^{1}$-fibration on $\wt S_{2}$ associated with $|H-E_{1}|$ induces 
the second projection of $\A^{1} \times \C^{*}$.
Since $\Delta$ is an isomorphism, we have $(C_{\wt S_{2}} \cdot H-E_{1})=1$.
Lemma \ref{lem:G2curve} now yields $C_{\wt S_{2}} \sim H-E_{5}$.
In particular, $\deg C=2$.

We have $(C_{\wt S_{2}} \cdot E_{5}-E_{6})=(C_{\wt S_{2}} \cdot H-E_{1}-E_{2}-E_{3})=1$. 
Hence $\Sing S_{2} \subset C$ and the linear hull $P \subset \P^{3}$ of $C$ is written as $\{z=ay\}$ for some $a \in \C$.
Thus $U$ is the affine modification of $\P^{3} \setminus S_{1} \cong \A^{3}_{(x,z,t)}$ with the locus 
\begin{align*}
(\{xt+xz^2+1=0, z=a\} \subset \{xt+xz^2+1=0\}).
\end{align*}
Therefore we have
\begin{align*}
U 
&\cong \{w(xt+xz^2+1)+z-a=0\} \subset \A^{4}_{(x, z, t, w)}\\
&\cong \{w(xt+1)+z=0\} \cong \A^{3}
\end{align*}
as desired.
\end{proof}

\begin{lem}\label{lem:G2deg3}
Suppose that the case \textup{(e)} of Theorem \ref{thm:main3tuple} holds with $(S_{1}, S_{2}) = (\{y=0\}, G_{2})$ and $\deg C \geq 3$.
Then $U \cong \A^{3}$.
\end{lem}


\begin{proof}
By Lemma \ref{lem:G2curve}, we obtain $C_{\wt S_{2}}  \sim 2H-E_{1}-E_{2}-E_{5}$ in Notation \ref{nota:normrat}.
In particular, $\deg C=3$.
Then $\Sing S_{2} \subset C$ since $(C_{\wt S_{2}} \cdot E_{2}-E_{3})=(C_{\wt S_{2}} \cdot E_{5}-E_{6})=1$.

Take a general point $p \in \la x, y\ra$ and $Q$ as the quadric surface containing $C$, $p$ and a general point of $\la y, z\ra$.
Then $Q$ contains $\la y, z\ra$ since $\sharp(\la y, z\ra \cap Q) \geq 3$.
Moreover, $Q$ contains $\la x, y\ra$ because Lemma \ref{lem:G2curve} shows that $\la x, y\ra$ is the unique curve of degree at most two containing $p$.
Hence $S_{2}|_{Q}=C+ \la x, y\ra + \la y, z\ra + L$ for some line $L \subset S_{2}$.
If $L= \la y,z \ra $, then $C \sim -K_{S_{2}}$, a contradiction.
Hence $S_{2}|_{Q}=C+ 2\la x, y\ra + \la y, z\ra$ and we can write 
$Q=\{xz+yl(x,y)=0\}$ for some linear form $l$.
Thus $U$ is the affine modification of $\P^{3} \setminus S_{1} \cong \A^{3}_{(x,z,t)}$ with the locus 
\begin{align*}
\left(
\left\{
xt+xz^2+1=0,
xz+l(x,1)=0
\right\} 
\subset 
\{xt+xz^2+1=0\} 
\right).
\end{align*}
Therefore we have
\begin{align*}
\ U \cong &\{w(xt+xz^2+1)+xz+l(x,1)=0\} \subset \A^{4}_{(x, z, t, w)}\\
\cong & \{w(xt+1)+xz+l(x,1)=0\}\\
\cong &\{w+xz+l(x,1)=0\}\cong \A^{3}.
\end{align*}
as desired.
\end{proof}

\begin{lem}\label{lem:G4deg2}
Suppose that the case \textup{(e)} of Theorem \ref{thm:main3tuple} holds with $(S_{1}, S_{2}) =  (\{y=0\}, G_{4})$ and $\deg C \leq 2$.
Then $U \cong \A^{3}$.
\end{lem}


\begin{proof}
As in the proof of Lemma \ref{lem:G2deg2}, the $\P^{1}$-fibration on $\wt S_{2}$ associated with $|H-E_{1}|$ induces an $\A^{1}$-bundle $S_{2} \setminus F \cong \A^{1} \times \C^{*} \to \C^{*}$.
Then $(C_{\wt S_{2}} \cdot H-E_{1})=1$ since $\Delta$ is an isomorphism.
Lemma \ref{lem:G4curve} now yields $C_{\wt S_{2}} \sim H-E_{i}$ for some $i \in \{5, 6\}$.
In particular, $\deg C=2$.

Since $(C_{\wt S_{2}} \cdot H-E_{1}-E_{2}-E_{3})=(C_{\wt S_{2}} \cdot E_{i})=1$, we have $\sharp (C \cap \mu_{*}(E_{i}))=2$.
Hence the linear hull $P \subset \P^{3}$ of $C$ satisfies $S_{2}|_{P}=C+\mu_{*}(E_{i})$.
On the other hand, we have $3\la y, z\ra =\{z=y\}|_{S_{2}} \in |-K_{S_{2}}|$.
Since we can not write $-K_{\wt S_{2}}-3E_{j}$ as a sum of $(-2)$-curves when $j=5$ or $6$, 
we obtain $\la y, z\ra_{\wt S_{2}} \sim E_{4}$.
Hence $\mu_{*}(E_{i})$ is either $\la x, y\ra$ or $\la x-z, y\ra$.

Suppose that $\mu_{*}(E_{i})=\la x, y\ra$.
Then $P=\{x=ay\}$ for some $a \in \C$ and 
$U$ is the affine modification of $\P^{3} \setminus S_{1} \cong \A^{3}_{(x, z, t)}$ with the locus 
\begin{align*}
\left(
\begin{array}{ll}
\{x=a, x^2-x^2z-x+xz^2+1-t+zt=0\} \\
\subset \{x^2-x^2z-x+xz^2+1-t+zt=0\}
\end{array}
\right).
\end{align*}
Hence 
\begin{align*}
U 
&\cong \{w((z-1)(t-x^{2}+x(z+1))+1)+x-a=0\} \subset \A^{4}_{(x, z, t, w)}\\
&\cong \{w(zt+1)+x=0\} \cong \A^{3}.
\end{align*}
The same conclusion can be drawn for the case where $\mu_{*}(E_{i})=\la x-z, y\ra$.
\end{proof}

\begin{lem}\label{lem:G4deg3}
Suppose that the case \textup{(e)} of Theorem \ref{thm:main3tuple} holds with $(S_{1}, S_{2}) =  (\{y=0\}, G_{4})$ and $\deg C \geq 3$. 
Then $U \cong \A^{3}$.
\end{lem}


\begin{proof}
As in the proof of lemma \ref{lem:G4deg2}, we have $(C_{\wt S_{2}} \cdot H-E_{1})=1$.
Lemma \ref{lem:G4curve} now yields $C_{\wt S_{2}} \sim 2H-E_{1}-E_{2}-E_{i}$ for some $i \in \{5, 6\}$.
In particular, $\deg C=3$.
Since $(C_{\wt S_{2}} \cdot E_{2}-E_{3})=(C_{\wt S_{2}} \cdot E_{i})=1$, we have $\Sing S_{2} \in C$ and $\sharp (C \cap \mu_{*}(E_{i}))=2$.

Take a general point $p \in \la y, z\ra$ and $Q$ as the quadric surface containing $C$, $p$ and a general point of $\mu_{*}(E_{i})$.
Then $Q$ contains $\mu_{*}(E_{i})$ since $\sharp(Q \cap \mu_{*}(E_{i})) \geq 3$.
Moreover, $Q$ contains $\la y, z\ra$ because Lemma \ref{lem:G4curve} shows that $\la y, z\ra$ is the unique curve of degree at most two containing $p$.
Hence $S_{2}|_{Q} = C+ \mu(E_{i}) + \la y, z\ra + L$ for some line $L \subset S_{2}$.
If $L$ is the other line in $S_{2}$, then we have $C \sim -K_{S_{2}}$, a contradiction.

We have checked that 
$\mu_{*}(E_{i})$ is either $\la x, y\ra$ or $\la x-z, y\ra$
in the proof of Lemma \ref{lem:G4deg2}.
Now suppose that $\mu_{*}(E_{i})=\la x,y \ra$.
Assume that $L=\la x,y \ra$ in addition.
Then we can write $Q = \{yl(x, y)+yt+zx=0\}$ for some  linear form $l$.
Now take $f$ as the automorphism of $\P^{3} \setminus S_{1} \cong \A^{3}_{(x, z, t)}$ such that $f^{*}(x) = x$, $f^{*}(z) = z-1$ and $f^{*}(t) = t -x^{2} + x(z+1)$.
Then we have
\begin{align*}
f(S_{2} \setminus F)
&=\{zt+1=0\} \cong \A^{1}_{(x)} \times \C^{*},\\
f(C \setminus (C \cap S_{1}))
&=
\left\{
zt+1=0,
l(x, 1)+t+x^{2}-x=0
\right\}
\end{align*}
in $\P^{3} \setminus S_{1} \cong \A^{3}_{(x, z, t)}$.
The projection $\mathrm{pr}_{t} \colon f(S_{2} \setminus F) \to \A^{1}_{(t)} \setminus \{0\}$ is the same as the second projection of $\A^{1} \times \C^{*}$.
Since the morphism $f(C \setminus (C \cap S_{1})) \to \C^{*}$ induced by $\mathrm{pr}_{t}$ is a dominant map of degree two, we have $\Delta = 2 \times \mathrm{id}_{\Z}$, which contradicts the assumption that $(C, S_{1}, S_{2})$ belongs to the case \textup{(e)}.

Therefore $S_{2}|_{Q} = C+ \la x,y \ra + 2\la y, z\ra$.
We can write $Q=\{yl(y,z)+x(z-y)=0\}$ for some linear form $l$.
Thus $U$ is the affine modification of $\P^{3} \setminus S_{1} \cong \A^{3}_{(x,z,t)}$ with the locus 
\begin{align*}
\left(
\left\{
\begin{array}{l}
x^2-x^2z-x+xz^2+1-t+zt=0,\\
l(1,z)+x(z-1)=0
\end{array}
\right\} 
\subset 
\left\{
\begin{array}{l}
x^2-x^2z-x+xz^2\\
+1-t+zt=0
\end{array}
\right\} 
\right).
\end{align*}
Therefore $U$ is isomorphic to
\begin{align*}
 &\{w((z-1)(t-x^2+x(z+1))+1)+l(1,z)+x(z-1)=0\} \subset \A^{4}_{(x, z, t, w)}\\
\cong & \{w(zt+1)+l(1,z+1)+xz=0\}\\
\cong &\{w+xz+l(1,z+1)=0\}\cong \A^{3}.
\end{align*}
The same conclusion can be drawn for the case where $\mu_{*}(E_{i}) = \la x-z, y\ra$.
\end{proof}

\begin{lem}\label{lem:R4}
Suppose that the case \textup{(e)} of Theorem \ref{thm:main3tuple} holds with $(S_{1}, S_{2}) =  (\{y=\gamma x\}, R_{4})$ for some $\gamma \in \C$. 
Then $U \cong \A^{3}$.
\end{lem}

\begin{proof}
Take $\psi \colon P \to \P^{3}$ as the blow-up along $E$.
Then $P$ has a $\P^{2}$-bundle structure $p \colon P \to \P^{1}$ associated with $|\psi^{*}\mc{O}_{\P^{3}}(1)-E_{\psi}|$.
Write $F_{t}$ as the $p$-fiber over $t \in \P^{1}$ and take $\infty \in \P^{1}$ as the point such that $F_{\infty}=(S_{1})_{P}$.
Since $(S_{2})_{P}$ is linearly equivalent to $\mc{O}_{\P^{3}}(3)-2E_{\psi}$, it is a sub $\P^{1}$-bundle.
Hence $(S_{2})_{P}=\wt S_{2}$.
We have $(S_{2})_{P} \setminus ((S_{2})_{P} \cap ((S_{1})_{P} \cup E_{\psi})) \cong \A^{1} \times \C^{*}$ by assumption and $p$ induces the second projection of $\A^{1} \times \C^{*}$.
Since $\Delta$ is an isomorphism, we have $(C_{P} \cdot \psi^{*}\mc{O}_{\P^{3}}(1)-E_{\psi})=1$.
Hence $C_{P}$ is a $p$-section.

On the other hand, $E_{\psi}|_{(S_{2})_{P}}$ is reducible by Lemma \ref{lem:2.16}.
Hence there is the unique point, say $0 \in \P^{1}$, such that $E_{\psi}|_{(S_{2})_{P}}$ is the sum of $\Sigma_{1}$ and $F_{0}|_{(S_{2})_{P}}$.
Since $\sharp(C_{P} \cap ((S_{1})_{P} \cup E_{\psi}))=\sharp(C \cap S_{1})=2$, we obtain $C_{P} \cap \Sigma_{1} \subset (S_{1})_{P} \cup F_{0}$.
Moreover, $E_{\psi}$ contains $C_{P} \cap F_{a}$ if $a=0$, and only if $a=0$ or $\infty$.

Now take $\chi \colon W \to P$ as the blow-up along $C_{P}$.
By construction, $U$ is the same as the affine modification of $P \setminus ((S_{1})_{P} \cup E_{\psi})$ with the locus $(C_{P} \setminus (C_{P} \cap ((S_{1})_{P} \cup E_{\psi})) \subset (S_{2})_{P} \setminus ((S_{2})_{P} \cap ((S_{1})_{P} \cup E_{\psi})))$.
Hence $U \cong U' \coloneqq W \setminus ((S_{1} \cup S_{2})_{W} \cup (E_{\psi})_{W})$.
On the other hand, we have a morphism $f \coloneqq (p \circ \chi)|_{U'} \colon U' \to \P^{1} \setminus \{ \infty \} \cong \A^{1}$.
By virtue of \cite[Main Theorem]{Kal}, it suffices to show that each $f$-fiber is isomorphic to $\A^{2}$ in order to prove $U \cong \A^{3}$.

Let $a \in \A^{1} \setminus \{0\}$ be a point.
Then $C_{P} \cap F_{a} \not \in E_{\psi}$.
Hence $f^{-1}(a)$ is the affine modification of $F_{a} \setminus (F_{a} \cap E_{\psi}) \cong \A^{2}$ with the locus 
$(C_{P} \cap F_{a} 
\subset 
(S_{2})_{P} \cap F_{a} \setminus ((S_{2})_{P} \cap E_{\psi} \cap F_{a}) \cong \A^{1})$, 
which is isomorphic to $\A^{2}$.

On the other hand, we have $C_{P} \cap F_{0} \in E_{\psi}$.
Hence $f^{-1}(0) \cong F_{0} \setminus (F_{0} \cap ((S_{2})_{P} \cup E_{\psi}))$.
Since $(S_{2})_{P}|_{F_{0}} = E_{\psi}|_{F_{0}}$, we obtain $f^{-1}(0) \cong \A^{2}$.
Hence we have the assertion.
\end{proof}

\begin{lem}\label{lem:R1}
Suppose that the case \textup{(b)} of Theorem \ref{thm:main3tuple} holds. 
Then $U \cong \A^{3}$.
\end{lem}

\begin{proof}
Take $\psi \colon P \to \P^{3}$, $p \colon P \to \P^{1}$, $F_{t}$ and $\infty \in \P^{1}$ as in Lemma \ref{lem:R4}.
Then $(S_{2})_{P}$ is a sub $\P^{1}$-bundle and hence $(S_{2})_{P}=\wt S_{2}$.
On the other hand, $E_{\psi}|_{(S_{2})_{P}}$ consists of $\Sigma_{3}$ and a non-reduced member of $|2f_{3}|$.
Hence there is the unique point, say $0 \in \P^{1}$, such that $E_{\psi}|_{(S_{2})_{P}}=\Sigma_{1}+2F_{0}|_{(S_{2})_{P}}$.

We note that $C_{P}$ is a $\psi$-section.
Since $\sharp(C_{P} \cap ((S_{1})_{P} \cup E_{\psi}))=\sharp(C \cap S_{1})=B_{2}(F)=2$, we obtain $C_{P} \cap \Sigma_{1} \subset (S_{1})_{P} \cup F_{0}$.
Moreover, $E_{\psi}$ contains $C_{P} \cap F_{a}$ if $a=0$, and only if $a=0$ or $\infty$.
Therefore analysis similar to that in the proof of Lemma \ref{lem:R4} shows that $U \cong \A^{3}$.
\end{proof}

\begin{lem}\label{lem:R2}
Suppose that the case \textup{(c)} of Theorem \ref{thm:main3tuple} holds. 
Then $U \cong \A^{3}$.
\end{lem}

\begin{proof}
Take $\psi \colon P \to \P^{3}$, $p \colon P \to \P^{1}$, $F_{t}$ and $\infty \in \P^{1}$ as in Lemma \ref{lem:R4}.
Then $(S_{2})_{P}$ is a sub $\P^{1}$-bundle and hence $(S_{2})_{P}=\wt S_{2}$.
On the other hand, $E_{\psi}|_{(S_{2})_{P}}$ consists of $\Sigma_{3}$ and a reduced member of $|2f_{3}|$.
Hence there are two point, say $0, 1 \in \P^{1}$, such that $E_{\psi}|_{(S_{2})_{P}}=\Sigma_{1}+F_{0}|_{(S_{2})_{P}}+F_{1}|_{(S_{2})_{P}}$.
We note that $\infty$ may coincide with $0$ or $1$.

By the choice of $C$, $C_{P}$ is a $\psi$-section.
Since $\sharp(C_{P} \cap ((S_{1})_{P} \cup E_{\psi}))=\sharp(C \cap S_{1})=B_{2}(F)+1$, we obtain $C_{P} \cap \Sigma_{1} \subset (S_{1})_{P} \cup F_{0} \cup F_{1}$.
Moreover, $E_{\psi}$ contains $C_{P} \cap F_{a}$ if $a=0$ or $1$, and only if $a=0, 1$ or $\infty$.
Therefore analysis similar to that in the proof of Lemma \ref{lem:R4} shows that $U \cong \A^{3}$.
\end{proof}

Combining Propositions \ref{prop:elldeg3}, \ref{prop:elldeg4}, \ref{prop:A2}, and Lemmas \ref{lem:G2deg2}--\ref{lem:R2}, we complete the proof of Theorem \ref{thm:mainisom}.

\section*{Acknowledgement}
The author is greatly indebted to Professor Hiromichi Takagi, his supervisor, for his encouragement, comments, and suggestions. 
He wishes to express his gratitude to Professor Takashi Kishimoto
for his helpful comments and suggestions. 
He also would like to express his gratitude to Doctor Akihiro Kanemitsu and Doctor Takeru Fukuoka for their helpful comments.

This work was supported by JSPS KAKENHI Grant Number JP19J14397 and 
the Program for Leading Graduate Schools, MEXT, Japan.


\begin{thebibliography}{KPR89}

\bibitem[AF03]{A-F}
Makoto Abe and Mikio Furushima.
\newblock On non-normal del {P}ezzo surfaces.
\newblock {\em Math. Nachr.}, 260:3--13, 2003.

\bibitem[AM75]{A-M75}
Shreeram~S. Abhyankar and Tzuong~Tsieng Moh.
\newblock Embeddings of the line in the plane.
\newblock {\em J. Reine Angew. Math.}, 276:148--166, 1975.

\bibitem[BD89]{Brenton-Drucker}
Lawrence Brenton and Daniel Drucker.
\newblock Perfect graphs and complex surface singularities with perfect local
  fundamental group.
\newblock {\em Tohoku Math. J. (2)}, 41(4):507--525, 1989.

\bibitem[BD93]{B-D}
S.~M. Bhatwadekar and Amartya~K. Dutta.
\newblock On residual variables and stably polynomial algebras.
\newblock {\em Comm. Algebra}, 21(2):635--645, 1993.

\bibitem[BL98]{B-L}
M.~Brundu and A.~Logar.
\newblock Parametrization of the orbits of cubic surfaces.
\newblock {\em Transform. Groups}, 3(3):209--239, 1998.

\bibitem[Bri68]{Brieskorn}
Egbert Brieskorn.
\newblock Rationale {S}ingularit\"{a}ten komplexer {F}l\"{a}chen.
\newblock {\em Invent. Math.}, 4:336--358, 1967/68.

\bibitem[BW79]{B-W}
J.~W. Bruce and C.~T.~C. Wall.
\newblock On the classification of cubic surfaces.
\newblock {\em J. London Math. Soc. (2)}, 19(2):245--256, 1979.

\bibitem[Fuj79]{Fuj79}
Takao Fujita.
\newblock On {Z}ariski problem.
\newblock {\em Proc. Japan Acad. Ser. A Math. Sci.}, 55(3):106--110, 1979.

\bibitem[Fuj82]{Fuj}
Takao Fujita.
\newblock On the topology of noncomplete algebraic surfaces.
\newblock {\em J. Fac. Sci. Univ. Tokyo Sect. IA Math.}, 29(3):503--566, 1982.

\bibitem[HW81]{H-W}
Fumio Hidaka and Keiichi Watanabe.
\newblock Normal {G}orenstein surfaces with ample anti-canonical divisor.
\newblock {\em Tokyo J. Math.}, 4(2):319--330, 1981.

\bibitem[Kal02]{Kal}
Sh. Kaliman.
\newblock Polynomials with general {$\bold C^2$}-fibers are variables.
\newblock {\em Pacific J. Math.}, 203(1):161--190, 2002.

\bibitem[Kis05]{Kis}
Takashi Kishimoto.
\newblock Compactifications of contractible affine 3-folds into smooth {F}ano
  3-folds with {$B_2=2$}.
\newblock {\em Math. Z.}, 251(4):783--820, 2005.

\bibitem[KPR89]{KPTR89}
Hanspeter Kraft, Ted Petrie, and John~D. Randall.
\newblock Quotient varieties.
\newblock {\em Adv. Math.}, 74(2):145--162, 1989.

\bibitem[KZ99]{K-Z}
Sh. Kaliman and M.~Zaidenberg.
\newblock Affine modifications and affine hypersurfaces with a very transitive
  automorphism group.
\newblock {\em Transform. Groups}, 4(1):53--95, 1999.

\bibitem[LPS11]{lps11}
Wanseok Lee, Euisung Park, and Peter Schenzel.
\newblock On the classification of non-normal cubic hypersurfaces.
\newblock {\em J. Pure Appl. Algebra}, 215(8):2034--2042, 2011.

\bibitem[MM82]{M-M}
Shigefumi Mori and Shigeru Mukai.
\newblock Classification of {F}ano {$3$}-folds with {$B_{2}\geq 2$}.
\newblock {\em Manuscripta Math.}, 36(2):147--162, 1981/82.

\bibitem[MS80]{M-S80}
Masayoshi Miyanishi and Tohru Sugie.
\newblock Affine surfaces containing cylinderlike open sets.
\newblock {\em J. Math. Kyoto Univ.}, 20(1):11--42, 1980.

\bibitem[Nag18]{Nag1}
Masaru Nagaoka.
\newblock Fano compactifications of contractible affine 3-folds with trivial
  log canonical divisors.
\newblock {\em Internat. J. Math.}, 29(6):1850042, 33, 2018.

\bibitem[Nag19]{Nag2}
Masaru Nagaoka.
\newblock On compactifications of affine homology 3-cells into quadric
  fibrations.
\newblock {\em arXiv preprint arXiv:1906.10626}, 2019.

\bibitem[Rei94]{Reid94}
Miles Reid.
\newblock Nonnormal del {P}ezzo surfaces.
\newblock {\em Publ. Res. Inst. Math. Sci.}, 30(5):695--727, 1994.

\bibitem[Spa66]{sp66}
Edwin~H. Spanier.
\newblock {\em Algebraic topology}.
\newblock McGraw-Hill Book Co., New York-Toronto, Ont.-London, 1966.

\bibitem[Suz74]{Suz74}
Masakazu Suzuki.
\newblock Propri\'{e}t\'{e}s topologiques des polyn\^{o}mes de deux variables
  complexes, et automorphismes alg\'{e}briques de l'espace {${\bf C}^{2}$}.
\newblock {\em J. Math. Soc. Japan}, 26:241--257, 1974.

\bibitem[tDP90]{tDP}
Tammo tom Dieck and Ted Petrie.
\newblock Contractible affine surfaces of {K}odaira dimension one.
\newblock {\em Japan. J. Math. (N.S.)}, 16(1):147--169, 1990.

\bibitem[Ye02]{Qia}
Qiang Ye.
\newblock On {G}orenstein log del {P}ezzo surfaces.
\newblock {\em Japan. J. Math. (N.S.)}, 28(1):87--136, 2002.

\end{thebibliography}
\end{document}